\documentclass[EJP,preprint]{ejpecp} % replace ECP by EJP if needed.
\usepackage[T1]{fontenc}
\usepackage[utf8]{inputenc}

\usepackage{color,latexsym,amsfonts,amssymb,amsmath, amsthm}
\usepackage{graphicx}
\usepackage{tikz}
\usetikzlibrary{patterns}
\usepackage{bbm}
\usepackage{dsfont}
\usepackage{nicefrac,enumerate}

\usepackage{todonotes}
\usepackage{filecontents}
\usepackage{subcaption}
\usepackage{comment}
\usepackage[normalem]{ulem}  % enables \sout  (strike through)
\usepackage[
backend=biber,
style=numeric,
maxbibnames=5,
giveninits,
isbn=false,
url=false
]{biblatex}
\renewbibmacro{in:}{\ifentrytype{article}{}{\printtext{\bibstring{in}\intitlepunct}}}   %Removes the strange `In:' ONLY in journal entries...

\SHORTTITLE{Recurrence versus Transience for Weight-Dependent Random Connection Models}

\TITLE{Recurrence versus Transience for Weight-Dependent Random Connection Models\support{We acknowledge support from DFG through the scientific network \emph{Stochastic Processes on Evolving Networks}. The research of CM is funded by the Deutsche Forschungsgemeinschaft (DFG, German Research Foundation) -- 443916008/SPP2265.}}
\AUTHORS{%
Peter~Gracar\footnote{Universit\"at zu K\"oln, Germany.
    \EMAIL{pgracar@math.uni-koeln.de,moerters@math.uni-koeln.de}}
\and Markus~Heydenreich\footnote{Ludwig-Maximilians-Universit\"at M\"unchen,
	Germany. \EMAIL{m.heydenreich@lmu.de}}
\and %% remove this line and below if single author
Christian~M\"onch\footnote{Johannes Gutenberg-Universit\"at Mainz,
   Germany. \EMAIL{cmoench@uni-mainz.de}}
\and Peter~M\"orters\footnotemark[2]
}%AUTHORS
 % Separate items with ;

\AMSSUBJ{60K35; 05C80} % Edit. Separate items with ;
\SUBMITTED{December 3, 2019} % Edit.
\ACCEPTED{January 30, 2022} % Edit.

\ARXIVID{1911.04350v3} % Edit.
\VOLUME{0}
\YEAR{2020}
\PAPERNUM{0}
\DOI{10.1214/YY-TN}

\ABSTRACT{
	We investigate random graphs on the points of a Poisson process in $d$-dimensional space, %$d\geq2$, 
	which combine
	scale-free degree distributions and long-range effects. Every Poisson point carries an independent random mark and given marks and 
	positions of the points we form an edge between two points independently with a probability depending
	via a kernel on the two marks and the distance of the points. Different kernels allow the mark to play different roles, like weight, radius or birth time of a vertex. The kernels depend on a parameter~$\gamma$, which determines the power-law exponent of the degree distributions. A further independent parameter $\delta$ characterises the decay of the connection probabilities of vertices as their distance increases.  We prove transience of the infinite cluster in the entire supercritical phase in regimes given by the parameters $\gamma$ and~$\delta$, and complement these results by recurrence results if $d=2$. Our results are particularly interesting for the soft Boolean graph model discussed in the preprint [arXiv:2108:11252] and the age-dependent random connection model recently introduced by Gracar et al.\ [Queueing Syst. 93.3-4 (2019)]}

\definecolor{halfgray}{gray}{0.55}
\definecolor{webgreen}{rgb}{0,.5,0}
\definecolor{webbrown}{rgb}{.6,0,0}
\definecolor{Maroon}{cmyk}{0, 0.87, 0.68, 0.32}
\definecolor{RoyalBlue}{cmyk}{1, 0.50, 0, 0}
\definecolor{Black}{cmyk}{0, 0, 0, 0}
\definecolor{pinkish}{RGB}{255, 192, 203}
\definecolor{orange}{rgb}{216, 64, 0}

\newcommand{\X}{\mathbf X}
   
\newcommand{\N}{\mathbb N}

\providecommand{\N}{\mathbb{N}}
\providecommand{\R}{\mathbb{R}}

\providecommand{\Rd}{{\mathds{R}^d}}

\newcommand{\Zd}{{\mathbb{Z}^d}}
\providecommand{\Z}{{\mathbb{Z}}}

\providecommand{\Q}{\mathbb{Q}}

\renewcommand{\P}{\mathbb{P}}

\newcommand{\E}{\mathbb{E}}

\renewcommand{\nicefrac}[2]{#1/#2}
\providecommand{\e}{\textup{e}}

\providecommand{\X}{X^{\deg}}

\providecommand{\g}{\gamma}

\providecommand{\e}{\varepsilon}

\renewcommand{\P}{\mathbb P}

\renewcommand{\phi}{\varphi}

\newcommand{\x}{\mathbf{x}}
\newcommand{\y}{\mathbf{y}}
\newcommand{\z}{\mathbf{z}}

\newcommand{\abs}[1]{\left| #1 \right|}

\begin{document}

%%%%%%%%%%%%%%%%%%%%%%%%%%%%%%%%%%%%%%%%%%%%%%%%%%%%%%%%%%%%%%%%%%%
%%                                                               %%
%% No need for \maketitle.                                       %%
%%                                                               %%
%%%%%%%%%%%%%%%%%%%%%%%%%%%%%%%%%%%%%%%%%%%%%%%%%%%%%%%%%%%%%%%%%%%

%%%%%%%%%%%%%%%%%%%%%%%%%%%%%%%%%%%%%%%%%%%%%%%%%%%%%%%%%%%%%%%%%%%
%%                                                               %%
%% Please replace what follows by the body of your article       %%
%% (up to the bibliography):                                     %%
%%                                                               %%
%%%%%%%%%%%%%%%%%%%%%%%%%%%%%%%%%%%%%%%%%%%%%%%%%%%%%%%%%%%%%%%%%%%

\section{Introduction and statement of results}\label{sec:intro}

%\PM{Title changed (slightly), draft abstract inserted, introduction rewritten. I use weight-dependent RCM for the general framework with varying kernels, for example in the title. Age-dependent remains what it was when introduced, referring only to the pa kernel.}
%\subsection{The weight-dependent random connection model}
In this paper we investigate the classical problem of transience versus recurrence of random walks on the infinite cluster of geometric random graphs in $d$-dimensional Euclidean space. We consider a general class of random graph models, which we call the \emph{weight-dependent random connection model}. This class contains classical models like the Boolean and random connection models as well as models that have long edges and scale-free degree distributions. The focus in this paper is on those instances  where the long-range or scale-free nature of graphs leads to new or even surprising results.
\smallskip

\pagebreak[3]
%\subsection{The weight-dependent random connection model}
%
%We study the \emph{weight-dependent random connection model}, which was introduced by Gracar, Grauer, Lüchtrath, and Mörters \cite{moerters2018} as an approximation to the local weak limit of the spatial preferential attachment model in Jacob and Mörters \cite{jacMor1}. 

The vertex set of the weight-dependent random connection model is a Poisson process of  unit intensity on $\R^d\times (0,1)$, for $d\geq1$. We think of a Poisson point $\x=(x,s)$ as a \emph{vertex} at \emph{position}~$x$ with \emph{mark} $s$.
% or \emph{weight} $s^{-1}$.} 
Two vertices $\x=(x,s)$ and $\y=(y,t)$ are connected by an edge with probability $\phi(\x,\y)$ for a connectivity function 
\begin{equation}\label{eq:phidef}
\phi\colon \big(\Rd\times(0,1)\big) \times\big(\Rd\times(0,1)\big) \to [0,1],
\end{equation}
satisfying $\phi(\x,\y)=\phi(\y,\x)$. Connections between different (unordered) pairs of vertices occur independently. We assume throughout that $\phi$ has the form 
\begin{equation}
\phi(\x,\y)=\phi\big((x,s),(y,t)\big)=\rho\big(g(s,t)|x-y|^d\big)
\end{equation}
for a non-increasing \emph{profile function} $\rho\colon \R_+\to[0,1]$ 
%with
%\begin{equation}\label{eq:rhoass}
%	\int_{\R^d}\rho(|x|)\,\textup{d}x=1 \quad\text{and}\quad
%	  \lim_{v\to\infty}\rho(v)\,v^{\delta} =1 \quad \text{for a parameter $\delta>1$}
%\end{equation}
and a suitable \emph{kernel} $g\colon(0,1)\times(0,1)\to\R_+$, which is
%. We assume further that $h$ is 
non-decreasing in both arguments. 
Hence vertices whose positions are far apart are less likely to be connected while vertices with small mark are likely to have many connections. 
%If the kernel function is of the form $h(s,t,v)=g(s,t) v^d$ 
We standardise the notation (without losing generality, see comment (i) in \cite[p.\ 312]{GracaGraueLuchtMorte19}) 
and assume that
\begin{equation}\label{eq:rhoass}
\int_{\R^d}\rho(|z|^d)\,\textup{d}z=1. 
\end{equation}
Then it is easy to see that the degree distribution of a vertex depends only on the kernel~$g$ (see for example \cite[Proposition 4.1]{GracaGraueLuchtMorte19}), but $\rho$ still has 
a massive influence on the likelihood of long edges in the graph. 
\smallskip

\pagebreak[3]
%\subsubsection*{The kernel function.} 
We next give concrete examples for the kernel $g$, and demonstrate that our setup yields a number of well-known models in continuum percolation theory. We define the functions in terms of parameters $\gamma\in[0,1)$ %(controlling the degree distribution) 
and $\beta\in(0,\infty)$. % (controlling the edge density). 
%and $\delta\in(1,\infty)$. 
The parameter $\gamma$ determines the strength of the influence of the vertex mark on the connection probabilities, large $\gamma$ correspond to strong favouring of vertices with small mark. 
In particular,  if $\gamma>0$ our models are scale-free with power-law exponent
$$\tau=1+\mbox{$\frac1\gamma$}.$$
The edge density is controlled by the parameter $\beta$, increasing $\beta$ increases the expected number of edges incident to a vertex at the origin. 
% $\beta$ measures the edge density -- this can be seen by applying substitution to the innermost integral in $\int_{[0,1]}\int_{[0,1]}\int_{\Rd}\rho(h(s,t,|x|))\,\textup{d} x\,\textup{d} s\,\textup{d} t$ which describes the expected number of edges connected to a vertex at $0$. 
Varying $\beta$ can also be interpreted as rescaling distances and is equivalent to varying the density of the underlying Poisson process. 
%This is most clearly seen for the product kernel. 
%Finally, $\delta$ describes how strong the geometry of the configuration of vertex positions influences the connection probabilities. Limiting cases arise for $\delta=1$, which essentially corresponds to a model without geometry, and $\delta=\infty$, which corresponds to a long-range percolation model with exponential decaying connection probabilities.
\medskip

\begin{itemize}
\item We define the \emph{plain kernel} as 
\begin{equation}\label{eqPlainKernel}
	g^{\rm plain}(s,t)=\tfrac1\beta.
\end{equation}
In this case we have no dependence on the marks, hence the model is not scale-free and we let $\gamma=0$.
If \smash{$\rho(r)={1}_{[0,a]}(r)$} for \smash{$a={{d}/{\omega_d}}$} and $\omega_d$ is the area of the unit sphere in $\R^d$, this gives the \emph{Gilbert disc model} with radius 
{$\frac12\sqrt[d]{\beta a}$}. Functions $\rho$ of more general form lead to the (ordinary) \emph{random connection model}, including in particular a continuum version of \emph{long-range percolation} when $\rho$ has polynomial decay at infinity. \smallskip
\item We define the \emph{sum kernel} as
\begin{equation}\label{eqSumKernel}
	g^{\rm sum}(s,t)=\tfrac1\beta\, \big(s^{-\gamma/d} + t^{-\gamma/d}\big)^{-d} 
\end{equation}
%\PM{Kernel changed to reflect "intersection of balls" interpretation. Similarly, min and max are the other way round compared to what we think intuitively.}
Interpreting $(\beta a s^{-\gamma})^{1/d}, (\beta a t^{-\gamma})^{1/d}$ as random radii
and letting $\rho(r)=1_{[0,a]}(r)$ we get the \emph{Boolean model} in which two vertices are connected by an edge if the associated balls intersect. The case of general profile function $\rho$ is referred to as \emph{soft Boolean model} 
in~\cite{GracaGraueMorte21}. We get a further variant of the model with the \emph{min-kernel} defined as 
$$g^{\rm min}(s,t)=\tfrac1\beta\, (s \wedge t)^\gamma .$$ 
Because $g^{\rm sum}\le g^{\rm min}\le2^dg^{\rm sum}$ the two kernels show qualitatively the same behaviour.\smallskip
\item For the \emph{max-kernel} defined as
$$g^{\rm max}(s,t)=\tfrac1\beta\, (s \vee t)^{1+\gamma},$$ 
we may choose any $\gamma>0$. This is a continuum version and generalization of the ultra-small scale-free geometric networks of Yukich~\cite{Yukic06}, which
is also parametrized to have power-law exponent $\tau=1+\frac1{\gamma}$. 
%	\MH{Perhaps we can mention right here that $h^{\rm min}\le h^{\rm min}\le2h^{\rm min}$ (as this explains why the two behave similarly)?}
\item
A particularly interesting  case is the \emph{product kernel}
\begin{equation}\label{eqProductKernel}
	g^{\rm prod}(s,t)=\tfrac1\beta\, s^\gamma t^\gamma,
\end{equation}
which leads to a continuum version of the \emph{scale-free percolation} model 
of Deijfen et al.\ \cite{DeijfHofstHoogh13,HeydeHulshJorri17}, see also  \cite{DepreHazraWuthr15,DepreWuthr19}. Here \smash{$s^{-\gamma}, t^{-\gamma}$} play the role of vertex weights. The model combines features of scale-free random graphs 
%(we have $\tau=1+\frac1\gamma$ as above) 
and polynomial-decay long-range percolation models
(for suitable choice of $\rho$).  
\smallskip
\item Our final example of a kernel $g$ is the \emph{preferential attachment kernel} 
%(henceforth \emph{pa-kernel})
\begin{equation}\label{eqPAKernel}
	g^{\rm pa}(s,t)=\tfrac1\beta\,(s\vee t)^{1-\gamma}(s\wedge t)^\gamma, 
\end{equation}
which gives rise to the \emph{age-dependent random connection model} introduced by Gracar et 
al.\ \cite{GracaGraueLuchtMorte19} as an approximation to the local weak limit of the spatial preferential attachment model in Jacob and M\"orters \cite{JacobMorte15}. 
In this model, $s$ and $t$ actually play the role of the birth times of vertices in the underlying dynamic network, we therefore refer to vertices with small $s$ %, equivalently with large weight, 
as old vertices. This model also combines scale-free degree distributions
%with power-law exponent $\tau=1+\frac1\gamma$ 
and long edges in a natural way.
%, but has a fundamentally different graph topology, as we will see. \CM{The different topology is entirely implicit if we remove the separate discussion of the robust case.}
\end{itemize}
\pagebreak[3]

\begin{table}[t]
\begin{center}
	\caption{Terminology of the models in the literature. }
	\label{tabModels}
	\begin{tabular}{llll} % <-- Alignments: 1st column left, 2nd middle and 3rd right, with vertical lines in between
		%      \textbf{Value 1} & \textbf{Value 2} & \textbf{Value 3}\\
		Vertices & Profile & Kernel & Name and reference\\
		\hline
		Poisson & indicator & plain  & Random geometric graph, Gilbert disc model \cite{Penro03}\\
		Poisson & general & plain & Random connection model \cite{Meest97}\\
		&&& Soft random geometric graph \cite{Penro16} \\
		lattice & polynomial & plain & Long-range percolation \cite{Berge02}\\
		Poisson & indicator & sum & (Poisson) Boolean model \cite{Hall85, Meest94}\\
		Poisson & polynomial & sum & {Soft Boolean model} \cite{GracaGraueMorte21}\\
		lattice & indicator & max & Ultra-small scale-free geometric networks \cite{Yukic06}\\
		Poisson & indicator & min & Scale-free Gilbert graph \cite{Hirsc17} \\
		lattice & polynomial & prod & Scale-free percolation \cite{DeijfHofstHoogh13,HeydeHulshJorri17} \\ 
		Poisson & polynomial & prod & Inhomogeneous long-range percolation \cite{DepreHazraWuthr15}\\
		&&&Continuum scale-free percolation \cite{DepreWuthr19}\\ 
		Poisson & general & prod & Geometric inhomogeneous random graphs \cite{BringmKeuscLengl19,KomjaLodew20}\\
		Poisson& general & pa & Age-dependent random connection model \cite{GracaGraueLuchtMorte19}\\
		%      continuum & infinite & plain & Random connection model \cite{MeestRoy96}\\
		\hline
	\end{tabular}
\end{center}
\end{table}
%MH{Left column could be renamed: Domain $\to$ Vertices, continuum $\to$ Poisson p.p.}

The weight-dependent random connection model with its different kernels has been studied in the literature under various names, we summarize some of them in Table \ref{tabModels}. A general framework in which models such as ours arise as limits of models on finite domains recently appeared in~\cite{HofstHoornMaitr21}.
We now focus on a profile function with polynomial decay 
\begin{equation}\label{eq:rhoass2}
\lim_{r\to\infty}\rho(r)\,r^{\delta} =1 \quad \text{for a parameter $\delta>1$,}
\end{equation}
and fix one of the kernel functions described above. % in \eqref{eqPAKernel}--\eqref{eqSumKernel} 
%as well as the parameters $\gamma$ and $\delta$, and 
We keep $\gamma, \delta$ fixed and study the resulting graph $\mathcal{G}^\beta$ as a function of $\beta$. Note that our assumptions $\delta>1$ and $\gamma<1$ imply that $\mathcal{G}^\beta$ is locally finite for all values of $\beta$, cf.~\cite[p.8]{GracaGraueLuchtMorte19}. 
We informally define $\beta_c$ as the infimum over all values of $\beta$ such that $\mathcal{G}^\beta$ contains an infinite component (henceforth the infinite \emph{cluster}); for a rigorous definition we refer to Section \ref{secModel}.  If \(d\geq 2\), we always have $\beta_c<\infty$, cf.~\cite{HeydeHulshJorri17}.
{General considerations show that there is at most one infinite cluster of $\mathcal G^\beta$. Indeed, 
it is established in \cite{GandolKeaneNewma92} that on $\Z^d$ there is at most one infinite cluster if the edge occupation measure is stationary and obeys the `finite energy property'; an analogous result for Poisson points applies in our case, cf.\ \cite{BurtoMeest93, Meest94}. Hence, we have that there is a unique infinite cluster whenever $\beta>\beta_c$.} We study the properties of this infinite cluster. 
\medskip

%\pagebreak[3]
Two cases correspond to different network topologies, see \cite{GracaLuchtMorte20, HeydeHulshJorri17, Yukic06}. If $\gamma>\frac12$ for the product kernel, or \smash{$\gamma>\frac{\delta}{\delta+1}$} for the preferential attachment, sum, or min kernel, or $\gamma>0$ for the max kernel,
we have $\beta_c=0$, i.e.\ there exists an infinite cluster irrespective of the edge density. We call this the \emph{robust case}. Otherwise, if $\gamma<\frac12$ for the product kernel or \smash{$\gamma<\frac{\delta}{\delta+1}$} for the preferential attachment, sum or min kernels, we have $\beta_c>0$ and call this the  \emph{non-robust case}.
\medskip

Our main interest is in whether the infinite cluster is recurrent (i.e., whether simple random walk on the cluster returns to the starting vertex almost surely), or transient (i.e., simple random walk on the cluster has positive probability of never returning to the starting vertex). 
Our results are summarized in the following theorem.\medskip 

\begin{theorem}[Recurrence vs.\ transience of the weight-dependent random connection model]
\label{thmRecurTrans}
{Consider the weight-dependent random connection model with profile function satisfying \eqref{eq:rhoass2} and assume that we are in the supercritical regime $\beta>\beta_c$.} 
%\ \\[-8mm]

\begin{enumerate}
	\item[(a)] For the preferential attachment kernel, the infinite component is almost surely %\eqref{eqPAKernel} 
		\begin{itemize}
			\item \emph{transient} if $1<\delta<2$ or $\gamma>\delta/(\delta+1)$;
			\item \emph{recurrent} in {$d=2$ if $\delta>2$ and $\gamma<1/3$}.
		\end{itemize}
	\item[(b)] For the min kernel and the sum kernel, the infinite component is almost surely 
	\begin{itemize} 
		\item \emph{transient} if $1<\delta<2$ or %$\delta>1$ and 
		$\gamma>\delta/(\delta+1)$;
		\item \emph{recurrent} in $d=2$ if  $\delta>2$ and $\gamma < 1/2$. %$\gamma < (\delta-1)/\delta$.
	\end{itemize}
	\item[(c)] For the product kernel, the infinite component is almost surely 
	\begin{itemize} 
		\item \emph{transient} if $1<\delta<2$ or %$\delta>1$ and 
		$\gamma>1/2$;
		\item \emph{recurrent} in $d=2$ if  $\delta>2$ and $\gamma < 1/2$.
	\end{itemize}
	\item[(d)] For the max
	kernel, the infinite component is almost surely \emph{transient} 
	{for all $0<\gamma<1$ and~$\delta>1$}. 
\end{enumerate}
\end{theorem}
%\pagebreak[3]

\captionsetup[subfigure]{labelfont=rm}
\begin{figure}[!h]
\centering
\begin{subfigure}{.45\textwidth}\centering 
	\begin{tikzpicture}
		% horizontal axis
		\draw[->] (0,0) -- (5,0) node[right] {$\gamma$};
		% labels
		\draw	(0,0) node[anchor=north] {0}
		(2,0) node[anchor=north] {\nicefrac{1}{2}}
		(4/3,0) node[anchor=north] {\nicefrac13}
		(4,0) node[anchor=north] {1};
		% ranges
		\draw	
		(0.75,4.5) node{{recur-}}
		(0.75,4) node{{rent}}
		(0.75,3.5) node{{for}}
		(0.75,3) node{{$d=2$}}
		(3.7,4.8) node{{\scriptsize $\gamma=\frac{\delta}{\delta+1}$}};
		%(2.8,5) node[above] {{\scriptsize $\gamma=\frac{\delta-1}{\delta}$}};
		
		% vertical axis
		\draw[->] (0,0) -- (0,5) node[above] {$\delta$};
		\draw (0,0) node[anchor=east] {1}
		(0,2.5) node[anchor=east] {2};
		
		\draw[dotted] (4,0) -- (4,5);
		% Us
		%\draw[thick] (0,2.5) -- (3.3,2.5);
		\draw[thick] (0,2.5) -- (2.7,2.5);
		\draw (1.9,1.5) node { transient}; %label
		%\draw (3.3,0.6) node {\scriptsize transient}; %label
		% curve
		\draw[thick] (4/3,2.5) -- (4/3,5);
%		\draw[dashed] (4/3,2.5) -- (4/3,0) node [below]{$1/3$};		
		\draw[dashed] (4/3,2.5) -- (4/3,0);		
		\draw[dashed] (0,0) plot[domain=0.5:0.75,variable=\g]({4*\g},{-2.5+2.5*\g/(1-\g)});
		\draw[thick] (0,0) plot[domain=2/3:0.75,variable=\g]({4*\g},{-2.5+2.5*\g/(1-\g)});
		%\draw[thick] (0,0) plot[domain=0.5:2/3,variable=\g]({4*\g},{-2.5+2.5*1/(1-\g)});
		%plot({3.8*\g+1/10},{1/3*(\g/(1-\g))+13/6});
		\draw[pattern=north west lines, pattern color=gray, draw=none] 
		(0,0) plot[smooth,samples=200,domain=2/3:0.75,variable=\g]({4*\g},{-2.5+2.5*\g/(1-\g)}) -- 
		plot[smooth,samples=200,domain=2/3:0.5,variable=\g]({4/3},{-2.5+2.5*1/(1-\g)});
	\end{tikzpicture}
	\caption[(a)]{Preferential attachment kernel}
\end{subfigure}
%\hskip{3em}
\begin{subfigure}{.45\textwidth}\centering 
	\begin{tikzpicture}
		% horizontal axis
		\draw[->] (0,0) -- (5,0) node[right] {$\gamma$};
		% labels
		\draw	(0,0) node[anchor=north] {0}
		(2,0) node[anchor=north] {\nicefrac{1}{2}}
		(4,0) node[anchor=north] {1};
		% ranges
		\draw	
		(1,3.5) node{{recurrent}}
		(1, 3) node{{for $d=2$}}
		(3.7,4.8) node{{\scriptsize $\gamma=\frac{\delta}{\delta+1}$}};
		%(2.8,5) node[above] {{\scriptsize $\gamma=\frac{\delta-1}{\delta}$}};
		
		% vertical axis
		\draw[->] (0,0) -- (0,5) node[above] {$\delta$};
		\draw (0,0) node[anchor=east] {1}
		(0,2.5) node[anchor=east] {2};
		
		\draw[dotted] (4,0) -- (4,5);
		% Us
		%\draw[thick] (0,2.5) -- (3.3,2.5);
		\draw[thick] (0,2.5) -- (2.7,2.5);
		\draw (2.3,1.5) node { transient}; %label
		%\draw (3.3,0.6) node {\scriptsize transient}; %label
		% curve
		\draw[thick] (2,2.5) -- (2,5);
		\draw[dashed] (0,0) plot[domain=0.5:0.75,variable=\g]({4*\g},{-2.5+2.5*\g/(1-\g)});
		\draw[thick] (0,0) plot[domain=2/3:0.75,variable=\g]({4*\g},{-2.5+2.5*\g/(1-\g)});
		%\draw[thick] (0,0) plot[domain=0.5:2/3,variable=\g]({4*\g},{-2.5+2.5*1/(1-\g)});
		%plot({3.8*\g+1/10},{1/3*(\g/(1-\g))+13/6});
		\draw[pattern=north west lines, pattern color=gray, draw=none] 
		(0,0) plot[smooth,samples=200,domain=2/3:0.75,variable=\g]({4*\g},{-2.5+2.5*\g/(1-\g)}) -- 
		plot[smooth,samples=200,domain=2/3:0.5,variable=\g]({2},{-2.5+2.5*1/(1-\g)});
		%%% horizontal axis
		%%\draw[->] (0,0) -- (5,0) node[right] {$\gamma$};
		%%% labels
		%%\draw	(0,0) node[anchor=north] {0}
		%%(2,0) node[anchor=north] {\nicefrac{1}{2}}
		%%(4,0) node[anchor=north] {1};
		%%% ranges
		%%\draw	
		%%%(1,3.5) node{{\scriptsize recurrent for}}
		%%%(1, 3) node{{\scriptsize $d=2$}}
		%%(1,3.5) node{{recurrent}}
		%%(1, 3) node{{for $d=2$}}
		%%(3.7,4.8) node{{\scriptsize $\gamma=\frac{\delta}{\delta+1}$}}
		%%(2,4.8) node {{\scriptsize $\gamma=\frac{\delta-1}{\delta}$}};
		%%
		%%% vertical axis
		%%\draw[->] (0,0) -- (0,5) node[above] {$\delta$};
		%%\draw (0,0) node[anchor=east] {1}
		%%(0,2.5) node[anchor=east] {2};
		%%
		%%\draw[dotted] (4,0) -- (4,5);
		%%% Us
		%%%\draw[thick] (0,2.5) -- (3.3,2.5);
		%%\draw[thick] (0,2.5) -- (2.7,2.5);
		%%\draw (2.3,1.5) node {transient}; %label
		%%%\draw (3.3,0.6) node {\scriptsize transient}; %label
		%%% curve
		%%%\draw[thick] (2,2.5) -- (2,5);
		%%\draw[dashed] (0,0) plot[domain=0.5:0.75,variable=\g]({4*\g},{-2.5+2.5*\g/(1-\g)});
		%%\draw[thick] (0,0) plot[domain=2/3:0.75,variable=\g]({4*\g},{-2.5+2.5*\g/(1-\g)});
		%%\draw[thick] (0,0) plot[domain=0.5:2/3,variable=\g]({4*\g},{-2.5+2.5*1/(1-\g)});
		%%%plot({3.8*\g+1/10},{1/3*(\g/(1-\g))+13/6});
		%%\draw[pattern=north west lines, pattern color=gray, draw=none] 
		%%    (0,0) plot[smooth,samples=200,domain=2/3:0.75,variable=\g]({4*\g},{-2.5+2.5*\g/(1-\g)}) -- 
		%%    plot[smooth,samples=200,domain=2/3:0.5,variable=\g]({4*\g},{-2.5+2.5*1/(1-\g)});
	\end{tikzpicture}
	\caption{Min kernel and sum kernel}
\end{subfigure}

\vskip\baselineskip
\begin{subfigure}{.45\textwidth}\centering 
	\begin{tikzpicture}
		% horizontal axis
		\draw[->] (6,0) -- (11,0) node[right] {$\gamma$};
		% labels
		\draw	
		(6,0) node[anchor=north] {0}
		(8,0) node[anchor=north] {\nicefrac{1}{2}}
		(10,0) node[anchor=north] {1};
		% ranges
		\draw  (7,3.5) node{{recurrent}};
		\draw   (7, 3) node{for $d=2$};
		%(3,4.5) node{{\scriptsize $\gamma=\frac{\delta}{\delta+1}$}};
		
		% vertical axis
		\draw[->] (6,0) -- (6,5) node[above] {$\delta$};
		\draw (6,0) node[anchor=east] {1}
		(6,2.5) node[anchor=east] {2};
		\draw[dashed] (8,0) -- (8,5);
		\draw[thick] (6.0,2.5) -- (8.0,2.5);
		\draw[dotted] (10,0) -- (10,5);
		\draw[thick] (8,2.5) -- (8,5);
		\draw (8.0,1.5) node { transient}; %label
	\end{tikzpicture}
	\caption{Product kernel}
\end{subfigure}
\begin{subfigure}{.45\textwidth}\centering 
	\begin{tikzpicture}
		% horizontal axis
		\draw[->] (6,0) -- (11,0) node[right] {$\gamma$};
		% labels
		\draw	(6,0) node[anchor=north] {0}
		(8,0) node[anchor=north] {\nicefrac{1}{2}}
		(10,0) node[anchor=north] {1};
		% ranges
		
		% vertical axis
		\draw[->] (6,0) -- (6,5) node[above] {$\delta$};
		\draw (6,0) node[anchor=east] {1}
		(6,2.5) node[anchor=east] {2};
		\draw[dotted] (10,0) -- (10,5);
		\draw (8.0,2.5) node { transient}; %label
		\draw[line width=1.5pt,dotted] (6,2.5) -- (6,4.9);
	\end{tikzpicture}
	\caption{Max kernel}
\end{subfigure}
\caption{Recurrent and transient regimes in Theorem \ref{thmRecurTrans}. The dashed line in (a),(b), (c) indicates the boundary of the robust regime. In the shaded areas of the top two diagrams the behaviour is unknown. The dotted line in (d) indicates recurrence for $d=2$, which follows trivially from (a).}
\label{figRecurTrans}	
\end{figure}
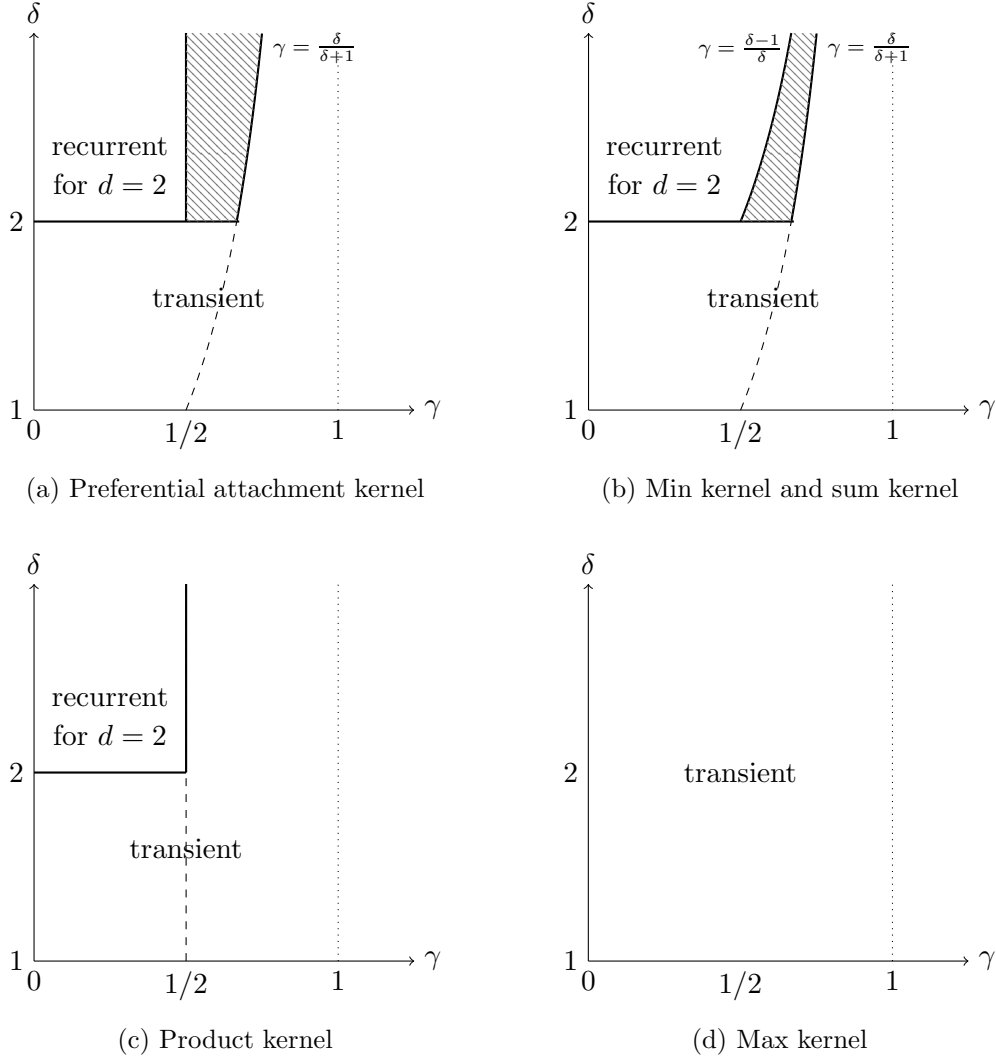
\

\pagebreak[3]
{\bf Remarks:} 
\begin{itemize}
\item \(\mathbf{ d=1 }\): Note that the kernel in (c) never induces a percolating graph in one dimension if $\delta>2$ and $\smash{\gamma<\frac12}$ \cite{DeijfHofstHoogh13} and the same is true for the kernels in (b) if $\delta>2$ and $\smash{\gamma<\frac{\delta-1}\delta}$ \cite{GracaLuchtMonch22}. In these cases all clusters are trivially recurrent. However, for the preferential attachment kernel in one dimension with $\delta>2$, the situation is less clear. It follows from results by Bode et al.\ \cite{BodeFountMuell15} for the `KPKVB-model', that the product kernel (which essentially coincides with the KPKVB-model after a change of coordinates) admits percolation for any $\rho$ which is non-increasing and positive in a neighbourhood of $0$, whenever \smash{$\gamma=\frac12$} and $\beta$ is sufficiently large. By monotonicity, it follows that the same is true for the preferential attachment kernel at any \smash{$\gamma\geq \frac12$}. We currently do not not know whether there are $\smash{\gamma< \frac{1}{2}}$ for which percolation can occur if $\beta$ is sufficiently large, see \cite{GracaLuchtMonch22} for a more detailed discussion.	
\end{itemize}\pagebreak

\begin{itemize}		
\item \(\mathbf{ d=2 }\). For (b) we conjecture that if $\delta>2$ the model is transient precisely if \smash{$\gamma>\frac{\delta-1}{\delta}$}, although our current proof only works if \smash{$\gamma>\frac{\delta}{\delta+1}$}, which by~\cite{GracaLuchtMorte20} is the robust regime. A heuristic argument for this conjectured behaviour is given in~\cite{GracaLuchtMonch22}. 
\item \(\mathbf{ d\geq 3 }\): We conjecture transience to hold in (a)--(c) also if $\delta\geq 2$ for all values of $\gamma>0$, as long as there is an infinite cluster. The analogous problem is open for the random connection model (and long-range percolation) in general. An analysis of this situation is beyond the scope of the present paper, and is postponed to future work. {Mind, however, that transience for \emph{sufficiently large} values of $\beta$ can be established by comparison with bond percolation on $\Z^d$.}
\item  \textbf{Any dimension}:
	\begin{itemize}
	\item {In (a)--(c) for $\delta<2$ our models dominate long-range percolation in the transient regime~\cite{Berge02}. This does not make our analysis redundant, as we show transience for all
		$\beta>\beta_c$ where $\beta_c$ may be strictly smaller than for the dominated model.}\smallskip
	\item For all kernels we consider, our investigation shows that robustness of the infinite cluster is sufficient for its transience, in particular this is the case in (d) for all values of $\gamma\in(0,1)$ and $\delta>1$. It would be interesting to ascertain this implication in greater generality. \smallskip
	\item  {We have excluded the boundary cases $\delta=2$ and \smash{$\gamma=\frac12$} from our main result, as in these cases the behaviour is dependent on fine details of the profile and kernel and therefore less suitable for the universal approach we develop. }	\smallskip					
	\item 
	Our results are in particular true for the plain kernel (corresponding to the case $\gamma=0$ in (b) and (c)), for similar results in this case see \cite{Sonme19}.	
	\end{itemize}
\end{itemize}

%\PM{When $\gamma=\frac12$ pa and product kernel agree. So we think that's non-robust and we need not exclude this.}

For a summary of the results we refer to Figure \ref{figRecurTrans}. {Arguably, the preferential attachment and sum kernels are the most interesting, but also technically the most involved models, and may therefore be considered as the main contribution of the present paper.} Indeed, the results for the product kernel are in correspondence with the findings of Heydenreich, Jorritsma and Hulshof \cite{HeydeHulshJorri17}. The behaviour at the critical point $\beta_c$ for the the plain kernel \eqref{eqPlainKernel} for $d\geq2$, has recently been investigated in \cite{HeydeHofstLastMatzk19} and, in the framework of long-range percolation, in \cite{Hutch21}. However, note that Berger \cite{Berge02} had previously shown that there is no infinite cluster in long range percolation at $\beta=\beta_c$ if $\delta<2$, a result that was later adapted to the product kernel \cite{DepreWuthr19} and that can be shown to remain true for our model whenever $\delta<2$, cf.\ the footnote on p.~\pageref{percfindiscussion}.
%\CM{Should we mention that we can deduce continuity of the percolation function in the weak decay regime just like in \cite{Berge02} for LRP? See end of transience section.} 
Different parametrisations have been used for the various models that can be treated in our framework, we have provided a translation in Table \ref{tabTranslation}.
%\CM{The papers mentioned in the text are not the same as the papers mentioned in the table caption.}
\smallskip

%Our results give a complete characterization of recurrence and transience of the infinite components in the models under consideration, with two exceptions: The behavior at the interfaces between the different regimes presumably depend on the finer asymptotice of $\phi$. 

\begin{table}[h]
\begin{center}
	\caption{Correspondence of parameters between our paper and \cite{DeijfHofstHoogh13, DepreHazraWuthr15, DepreWuthr19, HeydeHulshJorri17}.}
	\label{tabTranslation}
	\begin{tabular}{ccc} % <-- Alignments: 1st column left, 2nd middle and 3rd right, with vertical lines in between
		%      \textbf{Value 1} & \textbf{Value 2} & \textbf{Value 3}\\
		parameters in this paper && parameters in \cite{DeijfHofstHoogh13, DepreHazraWuthr15, DepreWuthr19, HeydeHulshJorri17}\\
		\hline\\[-2mm]
		$\beta$ &=& $\lambda^{d/2}$\\
		$\delta$ &=& $\alpha/d$\\
		$\gamma$ &=& $\frac d{\alpha(\tau-1)}$\\[2mm]
		\hline
	\end{tabular}
\end{center}
\end{table}
\ \\[-15mm]
	\paragraph{Overview of the paper.}
	Before we prove our results, we describe the model in a more rigorous manner in Section \ref{secModel}. In Section \ref{secTransience}, we treat the transience regimes indicated in Theorem~\ref{thmRecurTrans}. The recurrence results are established in Section~\ref{secRecurrence}. {Finally, a few basic results from electrical network theory, which are used frequently in our proofs, are collected in the appendix.} Throughout, we use the notation $f(x)\asymp g(x)$ if $f,g$ are positive functions such that $f(x)/g(x)$ are bounded away from $0$ and $\infty$.
%	\bigskip
%	
\pagebreak[3]

\section{The weight-dependent random connection model}\label{secModel}
\paragraph{Construction as a point process.} 	%on $(\R^d\times(0,1))^{[2]}\times(0,1)$.} \PM{Define $^{[2]}$ before using it.}
We give now a formal construction of the weight-dependent  random connection model. To this end, we enhance the construction given in \cite[Sections 2.1 and 2.2]{HeydeHofstLastMatzk19} by additional  vertex marks.
%(the birth times or \emph{inverse weights}}). 
For further constructions, see Last and Ziesche \cite{LastZiesc17} and Meester and Roy \cite{MeestRoy96}.  We construct the weight-dependent random connection model as a deterministic functional $\mathcal{G}_\phi(\xi)$ of a suitable point process $\xi$. 
Let $\eta$ denote a unit intensity $\Rd$-valued Poisson point process, which we can write as 
\begin{equation}
\eta=\{X_i\colon i\in \N\}; 
\end{equation}
such enumeration is possible by \cite[Corollary 6.5]{LastPenro18}. 
{In order to define random walks on the random connection model, it is convenient to have a designated (starting) vertex, and we therefore add an extra point $X_0=0$ when needed. This corresponds to working with a Palm version of the Poisson point process, which we denote by $\eta_0$.}\medskip

We further equip any Poisson point $X_i$, $i\in \N_0$, with an independent mark $S_i$ drawn uniformly from the interval \((0,1)\). 
{This defines a point process $\eta':=\{\X_i=(X_i,S_i)\colon i\in \N\}$ and $\eta_0':=\{\X_i=(X_i,S_i)\colon i\in \N_0\}$ on $\mathbb R^d\times (0,1)$, where here and throughout we write $\N_0=\N\cup\{0\}$.} 
Let $(\mathbb R^d\times (0,1))^{[2]}$ denote the space of all sets {$e\subset \mathbb R^d\times (0,1)$} with exactly two elements; these are the potential edges of the graph. 
We further introduce independent random variables $(U_{i,j}:i,j\in\N_0)$ uniformly distributed on the unit interval $(0,1)$ such that the double sequence $(U_{i,j})$ is independent of $\eta'$. Using $<$ for an arbitrarily fixed order on $\Rd$, we can now define 
\begin{equation}\label{eq:xidef}
\xi_0:=\big\{\big(\{(X_i,S_i),(X_j,S_j)\},U_{i,j}\big):X_i<X_j, i,j\in\N_0\big\}, 
\end{equation}
which is a point process on $(\mathbb{R}^d\times(0,1))^{[2]}\times(0,1)$. Similarly, we use $\xi$ to denote the configuration without the additional point at the origin. Mind that $\eta_0'$ might be recovered from~$\xi_0$. Even though the definition of $\xi_0$ 
{formally} depends on the ordering of the points of $\eta_0$, its distribution does not. 
%Given the independent marking $\xi$ of $\eta$, 
We now define the weight-dependent random connection model $\mathcal{G}_\phi(\xi)$ as a deterministic functional of $\xi$; its vertex and edge sets are given as 
\begin{align}
V(\mathcal{G}_\phi(\xi)) &= \eta', \\%= \{x\in\Rd: (\{(x,s),(y,t)\},u)\in\xi \text{ for some } y\ne x \text{ and } s,t,u\in[0,1]\},\\
E(\mathcal{G}_\phi(\xi)) &= \{ \{\X_i,\X_j\}\in V(\mathcal{G}_\phi(\xi))^{[2]}:X_i<X_j, U_{i,j} \leq \phi(\X_i,\X_j), i,j\in\N\}. 	\label{eqEdgeSet}
\end{align}
{In this section we have written $\mathcal{G}_\phi(\xi)$ in order to make the dependence on the connection function~$\phi$ explicit; in the following sections we will often fix a kernel function as well as the parameters $\delta$ and~$\gamma$, and write, just as in Section~\ref{sec:intro}, $\mathcal{G}^\beta=\mathcal{G}^\beta(\xi)$ or $\mathcal{G}^\beta_0=\mathcal{G}^\beta(\xi_0)$ if we wish to add the vertex at the origin. Furthermore, we use the notation $U_{\x,\y}$ for $U_{i,j}$ whenever $\x=\X_i$ and $\y=\X_j$, $i,j\in\N_0$,
to denote the edge marks.

\paragraph{The FKG inequality.}	
We further obtain a correlation inequality for increasing events known as FKG-inequality. We call a measurable function $f$ defined on point processes on $(\R^d\times(0,1))^{[2]}\times(0,1)$ \emph{increasing} if it is increasing in the underlying point process $\eta$ with respect to set inclusion, %further if $f$ is 
decreasing with respect to
{vertex} marks 
%(i.e., increasing with respect to vertex {weight}) 
and decreasing with respect to edge marks. 
%monotone in each component $(0,1)$ with respect to the ordinary order $\le$ on $\R$ and if it is 
Following arguments in \cite[Section 2.3]{HeydeHofstLastMatzk19}, we get for increasing functions $f_1,f_2$ that 
% we want that $\pla(\xi \in E \cap F) \geq \pla(\xi \in E)\pla(\xi \in F)$. Indeed, given two increasing (integrable) functions $f,g$, we have the more general statement
\begin{equation}\label{eq:FKG}
\begin{aligned}	\E[f_1(\xi) f_2(\xi)] & = \E \big[ \E[f_1(\xi)f_2(\xi) \mid \eta] \big] \geq \E \big[ \E[f_1(\xi)\mid \eta] \; \E[f_2(\xi)\mid \eta] \big]\\
& \geq \E[f_1(\xi)] \; \E[f_2(\xi)]. 
\end{aligned}
\end{equation} 
The first inequality uses the monotonicity properties of $\varphi$ (here we conditioned on the vertex set $\eta$), the second inequality is obtained through FKG for point processes, see e.g.\ \cite[Theorem 20.4]{LastPenro18}. Note that $\mathcal{G}_\phi(\cdot)$ itself is an increasing map with respect to the natural partial order on labelled graphs.
\pagebreak[3]
%\CM{In (12) edges are drawn if the weights are small, the same holds for the vertex weights, small weights create many edges. I'm not sure that 'montone' is enough, the 'direction' of the monotonicity needs to be checked for each component. Ideally, the version of FKG we obtain should reflect all 3 monotone operations in our model: decreasing vertex weights, decreasing edge weights, adding points to the Poisson process.  I also believe, the domain of $f$ should be larger (e.g. the point processes on $(\R^d\times(0,1))^{[2]}\times(0,1)$), and equipped with a different order structure: the inclusion order is to weak, because a marked point process $\xi_1$ only dominates $\xi_2$ via inclusion, if it contains $\xi_2$, i.e. it contains precisely the same marked points. But our model is monotone w.r.t. keeping the spatial points unchanged and decreasing the vertex weights -- this is not captured by the inclusion order.
	%}

\paragraph{Percolation.}
{Since $\mathbf 0:=(X_0,S_0)\in V(\mathcal{G}_\phi(\xi_0))$,  we write $\{\mathbf{0}\leftrightarrow\infty\}$ for the event that the random graph $\mathcal{G}_\phi(\xi_0)$ contains an infinite self-avoiding path $(\x_1,\x_2,\x_3,\dots)$ of vertices with $\x_i\in V(\mathcal{G}_\phi(\xi_0))$, $i\in\N$, such that $\{\mathbf 0,\x_1\},\{\x_1,\x_2\},\{\x_2,\x_3\}\ldots\in E(\mathcal{G}_\phi(\xi_0))$, and we say that in this case the graph \emph{percolates}}. 
{We denote the {percolation probability} by 
	\[ \theta(\beta)=\P\big(\mathbf{0}\leftrightarrow\infty \text{ in } \mathcal{G}_0^{\beta}\big),\]
	if kernel and profile are fixed. {Note that by ergodicity, we have $\theta(\beta)>0$ if and only if there exists an infinite cluster in $\mathcal{G}^{\beta}$ almost surely and that $\theta(\beta)$ is non-decreasing in $\beta$ (recall that the infinite cluster is unique).}
	This allows us to define the critical percolation threshold as
	\begin{equation}\label{eqDefBetaC}
		\beta_c:=\inf\{\beta>0: \theta(\beta)>0 \}\ge0.
	\end{equation}
	{Consequently, we call the model \emph{supercritical}, if $\beta>\beta_c$ and \emph{critical} if $\beta=\beta_c$.} 
\paragraph{Random walks.}
	We recall that, {as $\gamma<1<\delta$}, the resulting graph $\mathcal{G}_\phi(\xi_0)$ is locally finite, i.e.\ 
	\[\sum_{\y\in V(\mathcal{G}_\phi(\xi_0))}\mathds{1}\{\{\x,\y\}\in E(\mathcal{G}_\phi(\xi_0))\} <\infty\qquad\text{ for all $\x\in V(\mathcal{G}_\phi(\xi_0))$ almost surely},\]
	cf.\ \cite{GracaGraueLuchtMorte19}. 
	Given  $\mathcal{G}_\phi(\xi_0)$ %with $\mathbf{0} \leftrightarrow \infty$ 
	we define the \emph{simple random walk} on the random graph $\mathcal{G}_\phi(\xi_0)$ as the discrete-time Markov process $(Y_n)_{n\in\N}$ which starts at $Y_0=\mathbf{0}$ and has transition probabilities 
	\[  \mathsf{P}^{\mathcal{G}_\phi(\xi_0)}(Y_n=\y\mid Y_{n-1}=\x)=\frac{\mathds{1}\big\{\{\x,\y\}\in E(\mathcal{G}_\phi(\xi_0))\big\}}{\sum_{\z\in V(\mathcal{G}_\phi(\xi_0))}\mathds{1}\big\{\{\x,\z\}\in E(\mathcal{G}_\phi(\xi_0))\big\}},\]
	for $\x,\y\in V(\mathcal{G}_\phi(\xi_0)), n\in\N.$
	We say that $\mathcal{G}_\phi(\xi_0)$ is \emph{recurrent} if 
	$$\mathsf{P}^{\mathcal{G}_\phi(\xi_0)}\big(\exists\, n\ge1: Y_n=\mathbf{0}\big)=1,$$ otherwise we say that it is \emph{transient}. 
	%{Note that for $\beta>\beta_c$ recurrence and transience can be viewed as almost sure properties of the {\color{red}unique }infinite cluster in $\mathcal{G}_\phi(\xi)$. 
		%{\color{red} \sout{We say that the translation invariant model $\mathcal{G}_{\phi}(\xi)$ is transient if on the event $\{0\leftrightarrow \infty\}$ the graph $\mathcal{G}_{\phi}(\xi_0)$ is transient almost surely. Otherwise the graph $\mathcal{G}_{\phi}(\xi)$
				%is called recurrent.}} %We work with either configuration $\xi$ or $\xi_0$ on occasion, depending on whether we use translation invariance directly or not.
		{Moreover, we say that the connected component of $v\in\eta$ in $\mathcal{G}_{\phi}(\xi)$ is transient (resp.\ recurrent) if $\mathcal{G}_{\phi}(\Theta_{-v}\xi)$ is transient  (resp.\ recurrent) in the sense above (where $\Theta_{v}\eta(A)=\eta(A+v)$ is the spatial shift).}
		
		%\MH{Is it this what we want to say?}
		
		\section{Transience}\label{secTransience}
		%\color{gray}
		%We deduce the transience statements of Theorem~\ref{thmRecurTrans} from a more general result relating transience directly to the decay of $g$ at $0$ and of $\rho$ at $\infty$. More precisely, let
		%\begin{equation}\label{eq:kappadef}
		%\kappa(r)= \liminf_{v\,\to\, \infty}\frac{1}{\log v}\log \left( \int_{1/v^{1-r}}^{1-1/v^{1-r}}\int_{1/v^{1-r}}^{1-1/v^{1-r}} \rho(g(s,t)v)\,\textup{d}\,s\textup{d}t \right), \quad r\in [0,1),
		%\end{equation}
		%and note that $\kappa$ is non-increasing, hence $\kappa_0:=\lim_{r\searrow 0}\kappa(r)$ always exists and {does not depend} on $\beta$. Our analysis of transience depends crucially on the value of $\kappa_0.$
		%\color{black}
		{In this section, we prove the transience statements of Theorem~\ref{thmRecurTrans}. In Section~\ref{sec:tranrob}, we focus on the robust regime, i.e.\ $\gamma>0$ for max-kernel, $\gamma>1/2$ for product kernel, and $\gamma>\delta/(\delta+1)$ for preferential attachment, sum  and min-kernel. For  $\delta<2$ and general $\gamma>0$ we need a modified argument which is given in Section~\ref{sec:trannorob}. }

		\subsection{Transience in the robust regime}\label{sec:tranrob}
		Proving transient behaviour for the robust case hinges on a renormalisation sequence argument devised in its original form by Berger~\cite{Berge02} for the case of long-range percolation, which we now briefly sketch. The key idea is to check for the existence of a subgraph of the infinite component that branches sufficiently quickly as one considers the graph at larger and larger scales, which in turn yields that the subgraph, and consequently the infinite component, is transient. We start by considering a large but finite box of \(\R^d\) and choosing the vertex with the smallest mark inside the box. {We call this vertex \emph{dominant}. When the mark of this dominant vertex is sufficiently small, we consider the box \emph{good}.
			We next construct the quickly branching subgraph
			by considering ever larger scales and interpreting good boxes of a smaller scale as vertices while ignoring the remaining boxes.
			At each stage, we tile the current box into disjoint boxes of the previous size and check which of these boxes are good.}
		Note that they occur independently and with the same probability for all smaller boxes. Then, the bigger box is called good whenever a sufficiently large proportion of the boxes contained in it is good, these boxes are sufficiently well connected with each other and there exists a vertex in the newly constructed cluster with mark smaller than some even smaller predetermined value.
		By repeating this procedure at larger and larger scales we obtain a \emph{renormalised graph sequence} that is contained in the infinite component of the graph and can be shown to be transient using a fairly straightforward conductance argument as in \cite{Berge02}, cf. Lemma \ref{prop:renorm}. \medskip

		Before we formalise this argument, we first introduce some technical results that we use to prove our claims. We observe that the preferential attachment and sum kernels are both dominated by the min kernel 
		%whenever \(\gamma>\frac{1}{2}\). More precisely, for $\gamma>\frac{1}{2}$ and in particular $\gamma>\frac{\delta}{\delta+1}$} %\PM{Wieso nur unter dieser Voraussetzung? Das gilt immer!}
	{\[
		g^{\textrm{sum}}(s,t)\leq g^{\textrm{min}}(s,t)\quad\text{and}\quad g^{\textrm{pa}}(s,t)\leq g^{\textrm{min}}(s,t).
		\]
		Furthermore, $\gamma>\frac{\delta}{\delta+1}$ implies the robustness of the resulting graphs and in particular that $\beta_c=0$ for all the above mentioned kernels. Combined, this allows us to deduce the transience of the graphs obtained from the sum and preferential attachment kernels by only considering the min kernel and proving the transience thereof.}
	%\PM{Beide Ungleichungen sind allerdings verkehrt herum! Also: min dominiert!}
	\noindent
	%Since $\gamma>\frac{\delta}{\delta+1}$ implies robustness for all three kernels, we find that there is always an infinite cluster in the min kernel model. 
	%\PM{Klar, wenn es f\"ur alle drei gilt, dann auch f\"ur jedes von den drei. Aber das ist hoffentlich nicht, was Du wirklich sagen willst. Der Satz ist inhaltsleer.}
	%Consequently, this infinite cluster is dominated by the infinite clusters in the other two models and to obtain transience for the sum and preferential attachment kernels, it suffices to prove transience in the case of the min kernel.
	
	%\PM{Ich nehme an, dass wir ab sofort und bis auf weiteres den min Kern betrachten und nicht in jedem Lemma neu sagen m\"ussen, welche Kerne betroffen sind.}
	
	\paragraph{Two-connection probability bounds and other useful properties.}
	We start with
	the observation that in the case of the min kernel two vertices with sufficiently small marks are fairly likely to be connected via a vertex with a large mark, which we refer to as a \emph{connector}. This result is a variation of \cite[Lemma A.1]{HirscMonch19}.
	
	\begin{lemma}\label{lemma:two_connection}
		%Consider the weight-dependent random graph model with min kernel.
		{Given two vertices $\x=(x,s)$, $\y=(y,t)$ of ${\mathcal G}^\beta$ with $s,t\leq 1/2$} define
		\[
		k(\x,\y) = s^{-\gamma}\rho\Big(\beta^{-1} t^\gamma \big(s^{-\frac{\gamma}{d}} + \abs{x-y}\big)^d\Big)
		\]
		and
		\[
		q(\x,\y) = \frac{\rho(\beta^{-1})\kappa_d}{2}\left( k(\x,\y) \vee k(\y,\x) \right),
		\]
		where $\kappa_d$ is the volume of the $d$-dimensional unit ball. Then, with probability at least 
		\[
		1-e^{-q(\x,\y)},
		\] there exists {$\mathbf{z}=(z,u)\in\eta\times(1/2,1)$} which is a common neighbour of both $\x$ and $\y$.
	\end{lemma}
	{If vertices \(\x\) and \(\y\) share a common neighbour \(\z\) as above, then we call the intermediate vertex $\mathbf{z}=(z,u)$ a \emph{connector}, and we say that \(\x\) and  \(\y\) are connected \emph{through a connector}. 
	}
	
	\begin{proof}
		Let {$\mathcal X_c$} denote the set of common neighbours $\z=(z,u)$  of $\x$ and~$\y$ with mark $u>1/2$, that is, the vertices which satisfy \[U_{\x,\z} \leq \rho(\frac{1}{\beta}s^\gamma u^{1-\gamma} \abs{x-z}^d) \text{ and } U_{\y,\z} \leq \rho(\frac{1}{\beta}t^\gamma u^{1-\gamma} \abs{y-z}^d).\] Consider now the set $\mathcal X_c^{\x}$ of vertices $(z,u)\in \mathcal X_c$ with $\abs{x-z}^d \leq s^{-\gamma}$ that satisfy $U_{\x,\z} \leq \rho(1/\beta)$ and \smash{$U_{\y,\z} \leq \rho(\frac{1}{\beta}t^\gamma u^{1-\gamma} \abs{y-z}^d)$}. %Since $\abs{x-z}^d \leq s^{-\gamma}$ implies that $\rho(\frac{1}{\beta}s^\gamma u^{1-\gamma} \abs{x-z}^d) \geq \rho(1/\beta)$, $X_c^x$ is a subset of $X_c$.
		By the thinning theorem, $\mathcal X_c^{\x}$ forms a Poisson point process with total intensity
		\begin{equation}
			\begin{aligned}
				\int_{1/2}^1 \int_{B_{s^{-\gamma/d}}(x)}\rho(\beta^{-1})&\rho(\beta^{-1} t^{\gamma}\abs{{y}-z}^d)\, \textup{d}z\, \textup{d}u\\ 
				%&= \frac{\rho(1/\beta)}{2}\int_{B_{s^{-\gamma/d}}(x)}\rho(\beta^{-1} t^{\gamma}\abs{x-z}^d) \,\textup{d}z\\
				&\geq \frac{\rho(\beta^{-1})}{2}\int_{B_{s^{-\gamma/d}}(x)}\rho(\beta^{-1} t^{\gamma}(\abs{{x}-z} + \abs{x-y})^d)\, \textup{d}z\\
				&\geq \frac{\rho(\beta^{-1})\kappa_d}{2}s^{-\gamma}\rho(\beta^{-1} t^{\gamma}(s^{-\gamma/d} + \abs{x-y})^d).
			\end{aligned}\label{eq:lemma31}
		\end{equation}
		Hence,
		$
		\P(\mathcal X_c = \emptyset) \leq \P(\mathcal X_c^{\x} = \emptyset) \leq \exp\big( - \tfrac{\rho(\beta^{-1})\kappa_d}{2} k(\x,\y) \big).
		$
		Reversing the roles of $x$ and $y$ yields the stated result.
	\end{proof}\medskip
	
	We next state a technical result that allows us to draw conclusions on the mark distribution of the vertex with the smallest mark from a set of vertices, knowing that the largest mark among them is smaller than some given value. %This lemma is useful in order to draw conclusions about a large box of vertices, knowing that smaller boxes inside of it are ``well behaved''.
	
	\begin{lemma}\label{lemma:uniform}
		Let \((U_i)_{i
			\leq n}\) be a collection of independent on 
		$(0,1)$ uniformly distributed random variables. Then, for \(0\leq a<b\leq 1\), we have
		\begin{align}
			\P(\min_{i=1\dots n} U_i > a\,|\, \max_{i=1\dots n} U_i < b) \leq \e^{- na/b}.
		\end{align}
		%	Furthermore, for sufficiently large \(n\) 
		%	\begin{align}
			%		\P\big(U_i<(\sqrt{d}2^n((n+1)!)^2)^{-d/\gamma} \,|\, U_i<u_{n-2}\big)\leq \exp\{-cn\log (n)\},
			%	\end{align}
		%	where \(u_{n-2}\) is as defined above and \(c\) is a positive constant. 
		Furthermore, for $U$ uniformly distibuted on $(0,1)$ and 
		\(x<y\), 
		\begin{align}\label{eq:uniform_cond}
			\P(U<x\,|\,U<y)=\P(yU<x).
		\end{align}
	\end{lemma}
	
	\begin{proof}
		In order to prove the first bound, we use the simple fact that \(\P(U>a\,|\,U<b)=1-\frac{a}{b}\).
		%We begin now with the first bound. 
		Since \(1-x\leq \e^{-x}\), it is true that
		\begin{align*}
			\P(\min_{i=1\dots n} U_i > a\,|\, \max_{i=1\dots n} U_i < b) &=\big(1-\tfrac{a}b\big)^n\leq\e
			^{-na/b},
		\end{align*}
		which proves the claim. 
		The second statement of the lemma follows trivially from
		\begin{align*}
			\P(U<x\,|\,U<y)=\frac{\P(\{U<x\} \cap
				\{U<y\})}{\P(U<y)}=\frac{x}{y}=\P(yU<x).
		\end{align*}
		\ \\[-11mm]
	\end{proof}

	\paragraph{Proof of transience.}
	
	%We prove the results in two stages - first, we show that the statement holds when the edge density parameter \(\beta\) is sufficiently large. We do this by showing the existence of a renormalised graph sequence, which by \cite{NewmaSchul86,Berge02} implies that the underlying graph is transient. Next, we use a coarse graining argument to show that at a sufficiently large scale, one can observe transience even at smaller values of \(\beta\).
	We now formalise the definition of a renormalised graph sequence discussed earlier in this section.
	%Let a graph $G=(V,E)$ and a sequence $\{C_n\}_{n=1}^\infty$ be given and let $V_l(j_l,\dots,j_1)$ with $l \in \mathbb{N}_0$ and $j_n \in \{1,\dots,C_n\}$ be an element of the vertex set $V$. Furthermore, let \(V_0\) be some arbitrary vertex. Now let for $l \ge m >1$
	%\[
	%	\textstyle V_l(j_l,\dots,j_m) = \bigcup_{j_{m-1} =1}^{C_{m-1}} \dotsm \bigcup_{j_1=1 }^{C_1} V_l(j_l,\dots,j_1).
	%\]
	%We call the sets $V_l(j_l,\dots,j_m)$ \emph{bags}, and the numbers $C_n$ \emph{bag sizes.}
	
	\begin{definition}\label{def:grs} 
		We say that the graph $G=(V,E)$ is \emph{renormalised for the sequence} $(C_n)_{n\in\N}$ if we can construct an infinite sequence of subgraphs of $G$ such that 
		{\begin{itemize}
				\item the vertices of the \emph{$l$-stage subgraph} ($l\ge1$) are labelled by 
				$$V_{l}(j_l, \dots, j_1) \mbox{  for all $j_k \in \{1,\dots, C_k\}$ with $k=1,\dots,l$,} $$ 
				\item for \(l\geq m\) we set
				$$V_l(j_l,\dots,j_m)=\bigcup_{\substack{1\leq u_k\leq C_k\\ {\text{for } k=1\dots,m-1}}}V_l(j_l,\dots,j_m,u_{m-1},\dots u_{1}),$$
				\item  for every $l > m >2$, every $j_l,\dots,j_{m+1}$, and all pairs  $u_m, w_m \in \{1,\dots,C_m\}$ and $u_{m-1}, w_{m-1} \in \{1, \dots, C_{m-1}\}$ there is an edge in $G$ between a vertex in {the collection} $V_l(j_l, \dots, j_{m+1}, u_m, u_{m-1})$ and a vertex in {the collection} $V_l( j_l, \dots, j_{m+1}, w_m, w_{m-1}).$%\smallskip
				\item for completeness, a $0$-stage subgraph is a single vertex.
		\end{itemize}}
	\end{definition}
	%\MH{0-stage vertices?}
	\begin{figure}[!h]
		\begin{tikzpicture}[yscale=1.5, xscale=1.75]
			\draw [webgreen, dashed] plot [smooth cycle, tension=0.8] coordinates {(-2.5,-0.5) (-3,3) (3,3.5) (2.5, 0)};
			\draw[webgreen] (2,4) node{$V_n(1)$};
			
			\draw [red, dashed] plot [smooth cycle, tension=0.8] coordinates {(-1,0) (-3,2) (-1.5,3) (-0.5,2.5) (-1,1.5) (-0.7,0.8)};
			\draw[red] (-1,3.2) node{$V_n(1,1)$};
			\draw [red, dashed] plot [smooth cycle, tension=0.8] coordinates {(0.5,0) (3,1.6) (2.5,3) (1.5,2.5) (0.5,1.5) (0.3,0.8)};
			\draw[red] (2.2,3.2) node{$V_n(1,2)$};
			
			\draw [blue, dashed] plot [smooth cycle, tension=0.8] coordinates {(-1.2,0.3) (-2,1.2) (-1.2,1)};
			\draw[blue] (-2.3,2.5) node{$V_n(1,1,1)$};
			\draw [blue, dashed] plot [smooth cycle, tension=0.8] coordinates {(-1.5,1.5) (-2.5,2.5) (-1,2.5)};
			\draw[blue] (-2,1.2) node{$V_n(1,1,2)$};
			
			\draw [blue, dashed] plot [smooth cycle, tension=0.8] coordinates {(0.5,0.3) (2,1.2) (1.2, 1.6)};
			\draw[blue] (3,2) node{$V_n(1,2,1)$};
			\draw [blue, dashed] plot [smooth cycle, tension=0.8] coordinates {(1.5,2) (2.8,1.8) (2.1, 2.7)};
			\draw[blue] (2,1.3) node{$V_n(1,2,2)$};
			
			\draw[fill=black] (-1.4,0.8) circle (1pt);
			\draw[fill=black] (1,1) circle (1pt);
			\draw[fill=black] (-1.6,0.9) circle (1pt);
			\draw[fill=black] (2,2) circle (1pt);
			\draw (-1.4,0.8) to [bend left=20] (1,1);
			\draw (-1.6,0.9) to [bend left=20] (2,2);

			\draw[fill=black] (-1.3,2.1) circle (1pt);
			\draw[fill=black] (1.2,1.4) circle (1pt);
			\draw[fill=black] (-1.8,2.2) circle (1pt);
			\draw[fill=black] (2.3,2.4) circle (1pt);
			\draw (-1.3,2.1) to [bend left=20] (1.2,1.4);
			\draw (-1.8,2.2) to [bend left=20] (2.3,2.4);
			
			\draw[fill=black] (-1.25,0.55) circle (1pt);
			\draw[fill=black] (-1.6,2.05) circle (1pt);
			\draw (-1.25,0.55) to [bend right=20] (-1.6,2.05);
			
			\draw[fill=black] (1.4,1.2) circle (1pt);
			\draw[fill=black] (1.8,2.2) circle (1pt);
			\draw (1.8,2.2) to [bend right=20] (1.4,1.2);

		\end{tikzpicture}
		\caption{A renormalized graph sequence (stages \textcolor{blue}{$n-2$}, \textcolor{red}{$n-1$} and \textcolor{webgreen}{$n$}), where $C_n=C_{n-1}=C_{n-2}=2$
			{
				and $C_{n-3}=3$.} For $n=4$, the individual points in the picture can be understood as single vertices, whereas for $n>4$ they represent collections of vertices $V_n(1,1,1,1)$, $V_n(1,1,1,2)$, etc.}
		\label{fig:bags}
	\end{figure}
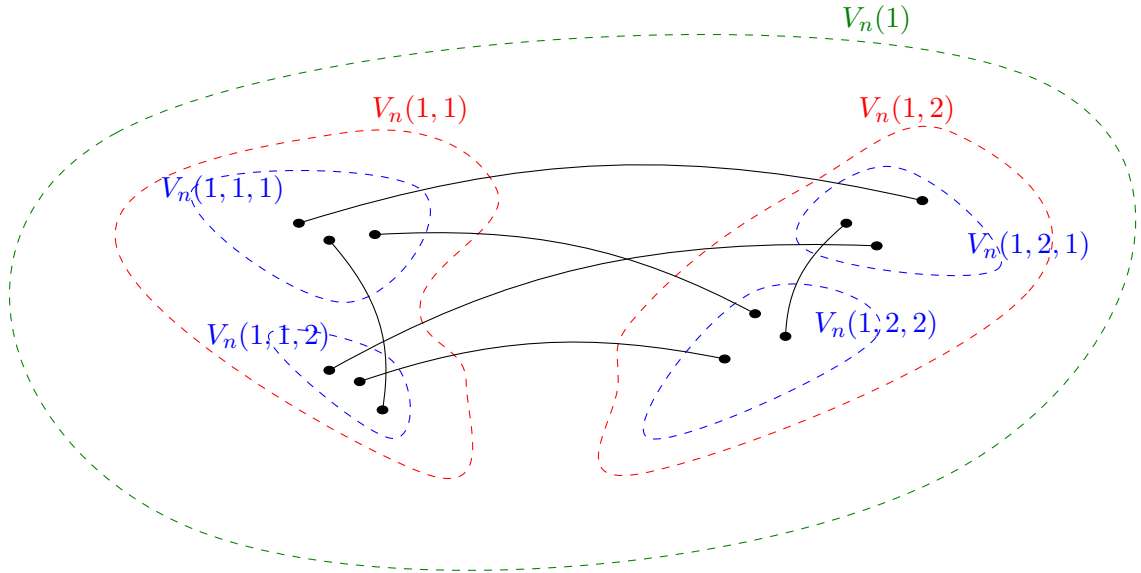
	
	{We may think of a renormalised graph as having a recursive structure: vertices are forming the 0-stage subgraphs, and every $n$-stage subgraph consists of a number 
		$C_n$ of $(n-1)$-stage subgraphs with the property that each pair of $(n-2)$-stage subgraphs in any of these is linked by an edge; see Fig.\ \ref{fig:bags}.} 
	%\rem{Renormalising a graph for $(C_n)_{n\in\N}$ amounts to providing a lower bound on the number of paths connecting a given vertex to $\infty$ in terms of the sequence $(C_n)$.}
	
	\begin{lemma}[{\cite[Lemma 2.7]{Berge02}}]\label{prop:renorm}
		A graph renormalized for the sequence $(C_n)_{n\in\N}$ is transient if $\sum_{n=1}^\infty C_n^{-1}<\infty$.
	\end{lemma}
	
	{We now prove transience for the min kernel when \(\gamma>\frac{\delta}{1+\delta}\) (i.e., we are in the robust case), and explain afterwards the modifications for the other kernels. }
	
	\begin{proposition}\label{prop:transience}
		%\lblue{Consider the random graph $\mathcal{G}^{\beta,p}$ obtained from the weight-dependent random connection model with min kernel and \(\gamma>\frac{\delta}{1+\delta}\) by independently percolating each vertex with fixed retention probability \(p\in(0,1]\). For all sufficiently large $\beta$, we have that $\mathcal{G}^{\beta,p}$ is transient.}
		
		{Let $\mathcal{G}^{\beta}$ be the weight-dependent random connection model with min kernel and \(\gamma>\frac{\delta}{1+\delta}\). Then, for all $\beta>0$, the infinite component of $\mathcal{G}^{\beta}$ is transient.}
	\end{proposition}
	%\PG{Wollen wir hier die potenziell mehrere Komponenten erwähnen oder einfach wie unten in Prop 3.6 und 3.7 ``$\mathcal{G}^{\beta}$ ist transient'' sagen?}
	%\PM{Nur ``$\mathcal{G}^{\beta}$ ist transient'' ist bisher definiert.}
	
	%\CM{Suggestion: Hier ohne $p$ arbeiten, sondern a posteriori $\mathcal G^{\beta,p}$ betrachten und offenkundige scale-invariance benutzen.}
	%\PG{Eher nein, weil man muss im Beweis sehen, dass man $\beta$ uniform in $p$ wählen kann, was ohne $p$ mitzuschleppen nicht offensichtlich gilt. Die scale-invariance gilt nicht, wenn wir zum coarse-graining springen und uns auf dem Beweis beziehen.}
	
	Our proof is inspired by the arguments in \cite[Prop.\ 5.3]{HeydeHulshJorri17}. However, our proof has the necessary coarse graining ideas  built into the construction and therefore yields the result for all values of~$\beta$ directly.
	
	\begin{proof}
		Begin by choosing a positive constant $\varepsilon$ smaller than $2(\delta+1)\gamma d-2\delta d$, which is possible since $\gamma>\frac{\delta}{\delta+1}$ implies this expression is strictly positive.
		Next, let $n_*$ be a large constant whose value we will fix later in the proof. Define, 
		for $n\geq 1$,
		\begin{align}\label{eq:un}
			u_n&=\frac{1}{c_1}(n_*+n)^{-\frac{\varepsilon}{\gamma(\delta+1)}}2^{-\frac{(n_*+n)d\delta}{n_*\gamma(\delta+1)}}\Big(\tfrac{(n_*+n)!}{n_*!}\Big)^{-\frac{2d\delta}{\gamma(\delta+1)}},
		\end{align}
		where \smash{$c_1=(\frac{1}{2}\kappa_d\rho({1}/{\beta})\beta^{\delta-1} d^{-d\delta/2})^{-1/\gamma(\delta+1)}$} is a positive constant, and $\kappa_d$ is the volume of the $d$-dimensional unit ball.
		Let $B$ be the event that in a given box of side-length $D_1:=2(n_*+1)^2$ there exists a vertex with mark smaller than $u_1$ and define $p_B$ to be the probability of this event. A simple calculation using the properties of Poisson point processes yields that there exists a positive constant $c$ for which
		\begin{align*}
			p_B&=1-\exp\{-u_1D_1^d\}
			\geq 1-\exp\Big\{-c n_*^{-\frac{\varepsilon}{\gamma(\delta+1)}-\frac{2d\delta}{\gamma(\delta+1)}+2d}\Big\}
		\end{align*}
		and, by choice of $\varepsilon$, the probability $p_B$ can therefore be made arbitrarily close to $1$ by choosing~$n_*$ large. Although we will not explicitly highlight this throughout most of the proof, we work from here on under the assumption that $n_*$ is sufficiently large for $p_B$ to be close to $1$ (say, greater than $3/4$) so that the quantities where $p_B$ appears are all of the same order of magnitude as if $p_B$ was simply equal 1. 
		We define for $n\geq 2$ the sequences
		\begin{align*}
			C_n&:=p_B(n_*+n)^{2d},\quad D_n:=2(n_*+n)^2.
		\end{align*}
		Next, we partition \(\R^d\) into disjoint boxes of side length \(D_1\); we call them \emph{\(1\)-stage boxes}. We now define the renormalization procedure. We partition \(\R^d\) again, grouping \(D_2^d\) \(1\)-stage boxes together to form \emph{\(2\)-stage boxes}. We continue like this for all \(n\geq 3\), so that the \(n\)-stage boxes represent a partitioning of \(\R^d\) into boxes of side length \(\prod_{i=1}^n D_i\).
		\medskip
		
		We now define what it means for a box to be ``good'' or ``bad'', starting with \(1\)-stage boxes. 
		We declare a \(1\)-stage box as \emph{good} if it contains at least one vertex with mark smaller than \(u_1\), and for each good \(1\)-stage box we declare the vertex with the smallest mark to be \emph{\(1\)-dominant}. We define $L_1(v)$ to be the event that the $1$-stage box centered around $v$ is good, and omit $v$ when considering the box containing the origin.
		We declare a \(2\)-stage box as \emph{good} if it contains at least $C_2$ good \(1\)-stage boxes and among the corresponding \(1\)-dominant vertices at least one has mark smaller than $u_2$. As before, we define $L_2(v)$ to be the event that the $2$-stage box centered around $v$ is good, and omit $v$ when considering the box containing the origin.\medskip
		
		For \(n\geq 3\), let $\mathcal{Q}$ be the set of all good $(n-1)$-stage subboxes of the $n$-stage box $Q$. We declare \(Q\) as \emph{good} if the following three conditions hold:
		\begin{description}
			\item[(E)] $\mathcal{Q}$ contains at least $C_n$ boxes;\smallskip
			\item[(F)] for any pair of  boxes $Q', Q''\in \mathcal Q$
			every pair of distinct \((n-2)\)-dominant vertices in \(Q'\) 
			and~\(Q''\) is connected through a connector;
			\smallskip
			\item[(G)] at least one of the boxes $Q'\in\mathcal Q$
			contains an \((n-1)\)-dominant vertex with mark no larger than \(u_n\).
		\end{description}
		We declare for each good \(n\)-stage box the vertex with the smallest mark as the \(n\)-dominant vertex. 
		%, given that its mark is smaller than \(u_n\).
		We now define \(E_n(v)\), \(F_n(v)\) and \(G_n(v)\) to be the events that conditions \((E)\), \((F)\) and \((G)\) hold for the \(n\)-stage box containing vertex \(v\). When considering the origin, we omit the vertex in this notation. We do the same for the event \(L_n(v)\), which we define to be the event that the corresponding \(n\)-stage box is good. %Before proceeding, note that the events $E_n$, $F_n$ and $G_n$ are all increasing in the vertex marks.
		\medskip
		
		Due to translation invariance it is enough to show that
		\begin{align*}
			\P\Big(\bigcap_{n=1}^\infty L_n\Big)>0
		\end{align*}
		in order to prove our claim.
		First, note that 
		\begin{align*}
			&\P\Big(\bigcap_{n=1}^\infty L_n\Big)
			=1-\P\Big(\bigcup_{n=1}^\infty L_n^c\Big)
			\geq 1-\sum_{n=1}^\infty\P(L_n^c).
		\end{align*}
		Therefore, it suffices to show that the sum on the right can be made smaller than $1$.
		\pagebreak[3]
		
		For $n=1$, we already have that $\P(L_n^c)=1-p_B$ which can be made arbitrarily small by setting $n_*$ large. For $n\geq 2$ we decompose the event $L_n^c$ with respect to $E_n$, $F_n$ and $G_n$ and obtain that
		\begin{align}\label{eq:lnc}
			\P(L_n^c)\leq\P(E_n^c)+\P(F_n^c\,|\,E_n)+\P(G_n^c\,|\,E_n),
		\end{align}
		where we set $\P(F_2^c\,|\,E_2)$ as $0$ in order to simplify notation.\medskip
		
		We \emph{first} bound \(\P(F_n^c\,|\,E_n)\) for $n\geq 3$. 
		To do that, let $\mathcal{X}$ be the restriction of the marked point process $\eta'$ to vertices with marks smaller than $1/2$. Then, given $\mathcal{X}$,  the event $E_n$ and,
		for two vertices in $\mathcal{X}$,
		the existence of a vertex in $\eta'\setminus\mathcal{X}$  connecting them are increasing events depending only on the edge marks and on the marked vertices of $\eta'\backslash\mathcal{X}$. Thus, given $\mathcal{X}$, they are positively correlated.
		Next, note that any two vertices in the same \(n\)-stage box are at most
		\begin{align*}
			\sqrt{d}\prod_{k=1}^n D_k=\sqrt{d}2^n\Big(\tfrac{(n_*+n)!}{n_*!}\Big)^2
		\end{align*}
		away from each other and similarly any $(n-2)$-dominant vertices have mark smaller than $u_{n-2}$.
		Therefore, using~Lemma \ref{lemma:two_connection}  the conditional probability given $\mathcal{X}$ that two $(n-2)$-dominant vertices of $\mathcal{X}$ belonging to the same $n$-stage box are not connected through a connector is smaller than
		\begin{align*}
			\exp\Big\{-\tfrac{1}{2}\rho\Big(\tfrac{1}{\beta}\Big)\kappa_d\beta^{\delta}u_{n-2}^{-\gamma(\delta+1)}\Big((s\vee t)^{-\gamma/d}+\sqrt{d}2^n\big(\tfrac{(n_*+n)!}{n_*!}\big)^2\Big)^{-d\delta}\Big\},
		\end{align*}
		where we highlight that the bound is uniform in the locations of the two vertices and that the \((s\vee t)^{-\gamma/d}\) term is unbounded from above.
		% We have also used that under $\mathcal{X}$, the event that two such vertices are connected through a connector is positively correlated to the event $E_n$ in order to justify using \Cref{lemma:two_connection} despite $E_n$ also revealing information about the existence of connectors. 
		We therefore consider the cases when \(s\wedge t\) is bigger or smaller than \smash{$(\sqrt{d}2^n((n_*+n)!/{n_*!})^2)^{-d/\gamma}$}
		and obtain with the help of Lemma \ref{lemma:uniform} that the conditional probability given $\mathcal{X}$ of two $(n-2)$ dominant vertices belonging to the same $n$-stage box being connected to a common vertex of $\eta'\backslash\mathcal{X}$ is smaller than
		\begin{align*}
			%\P(A_n\,|\,{\mathcal{E}})&\leq\P\big(A_n\,|\,{\mathcal{E}\cap\sigma(}s\wedge t\geq(\sqrt{d}2^n((n_*+n)!/{n_*!})^2)^{-d/\gamma})\big)\\
			%&\quad+\P\big(s\wedge t<(\sqrt{d}2^n((n_*+n)!/{n_*!})^2)^{-d/\gamma}{\,|\,\mathcal{E}}\big)\\
			\exp\Big\{-\tfrac{1}{2}\rho({\beta}^{-1})\kappa_d\beta^{\delta}u_{n-2}^{-\gamma(\delta+1)}d^{-\frac{d\delta}{2}}2^{-nd\delta}(\tfrac{(n_*+n)!}{n_*!})^{-2d\delta}\Big\}+\exp\{-cn\log(n_*+n)\}\\
			\leq\exp\{-\beta (n_*+n)^\varepsilon\}\vee \exp\{-\tilde cn\log(n_*+n)\},\end{align*}
		where \(\tilde c\) is a positive constant and we used that \(\log (n!)\asymp n \log n\). 
		Next, there are 
		\begin{align*}
			{D_n^d D_{n-1}^d\choose 2 }<4^d(n_*+n)^{4d}
		\end{align*}
		possible pairs of {distinct}  \((n-2)\)-stage boxes in an \(n\)-stage box and therefore at most \(4^d(n+1)^{4d}\) connections via a connector vertex. Conditionally on $\mathcal{X}$, the events $F_n$ and $E_n$ are increasing with respect to points, edges and vertex marks which are not determined by the configuration~$\mathcal{X}$. Consequently, conditionally on $\mathcal{X}$, $F_n^c$ and $E_n$ are negatively correlated. By taking the union bound we thus obtain 
		\begin{align}\label{eq:fnc}
			\P_\mathcal{X}(F_n^c\,|\, E_n)\leq\exp\{d\log(4)+4d\log(n_*+n)-\beta (n_*+n)^\varepsilon\wedge \tilde cn\log(n_*+n)\},
		\end{align}
		where we write $\P_\mathcal{X}$ for the conditional expectation given $\mathcal{X}$. 
		Taking the expectation and using that this bound is uniform in $\mathcal{X}$ yields the same bound for the unconditional probability.\medskip
		%\begin{align}\label{eq:fnc}
		%	\P(F_n^c\,|\, E_n)\leq\exp\{d\log(4)+4d\log(n_*+n)-\beta 
		%(n_*+n)^\varepsilon\wedge cn\log(n_*+n)\}.
		%\end{align}
		
		Next, we bound \(\P(G_n^c\,|\,E_n)\). For some positive constant \(c_2\) we write
		\begin{align*}
			\tfrac{u_n}{u_{n-1}}&=\big(\tfrac{n_*+n}{n_*+n-1}\big)^{-\frac{k}{\gamma(\delta+1)}}2^{-\frac{d\delta}{n_*\gamma(\delta+1)}}(n_*+n)^{-\frac{2d\delta}{\gamma(\delta+1)}}
			\geq c_2(n_*+n)^{-\frac{2d\delta}{\gamma(\delta+1)}}
		\end{align*}
		and therefore using Lemma \ref{lemma:uniform} we get
		\begin{align}
			\P(G_n^c\,|\,E_n)&\leq \exp\Big\{-C_n\tfrac{u_n}{u_{n-1}}\Big\}\leq \exp\Big\{-c_2p_B(n_*+n)^{\frac{2d(\gamma(\delta+1)-\delta)}{\gamma(\delta+1)}}\Big\},\label{eq:gnc}
		\end{align}
		where the exponent of the term \(n_*+n\) is positive whenever $\gamma>\frac{\delta}{\delta+1}$. To keep things concise, we will refer to this exponent as $k$ from now on.\medskip
		
		%{\color{blue}\sout{In order to bound the remaining term of \eqref{eq:lnc} we note first that $L_{n-1}(v)$ are increasing events for all $v\in\R^d$ and are therefore positively correlated. Furthermore, the distribution of $L_{n-1}(v)$ is independent of $v$. Therefore, using a standard coupling argument we have that the $L_{n-1}(v)$ stochastically dominate random variables with the same  distribution but sampled independently. Since $E_n$ is an increasing event, it suffices to bound the probability of $E_n^c$ in the coupled model.}}
		
		%\PG{Alternative below. What is not explicitly mentioned is that even the ``worst case'' events are still correlated (luckily positively correlated, since the correlations all come only from the existence of connectors, which do not have to lie in the individual box). For the generic $L_{n-1}$ it is true that the event is not monotone (so blue above is not quite right), the ``worst case'' version however is since once the event occurs, increasing any of the marks or adding new points cannot negate the event.}
		
		In order to bound the remaining term of \eqref{eq:lnc} we will again use $\mathcal{X}$. Crucially, given $\mathcal{X}$ the events $L_{n-1}(v)$ are all positively correlated. To see why, note that for any fixed configuration of vertices from $\mathcal{X}$, the event $L_{n-1}(v)$ depends only on the realization of the edge marks and the marked vertices in $\eta'\setminus\mathcal{X}$, i.e. the \emph{connector} vertices. Since the number of required edges and vertices from $\eta'\backslash \mathcal{X}$ does not change for a given realization of $\mathcal{X}$, all $L_{n-1}(v)$ are increasing events and therefore positively correlated. Consequently, given $\mathcal{X}$ and using a standard coupling argument, the collection of events  $L_{n-1}(v)$ stochastically dominates events with the same marginal distribution but sampled independently.
		We obtain
		\[
		\mathbb{P}(E_n^c)=\E[\mathbb{P}_\mathcal{X}(E_n^c)]\leq \E[\mathbb{P}_\mathcal{X}(\tilde E_n^c)]=\mathbb{P}(\tilde E_n^c),
		\]
		where we denote by $\tilde E_n^c$ the event that there are at most $C_n$ good $(n-1)$-stage boxes in the $n$-stage box, with each event $L_{n-1}(v)$ sampled independently. We can therefore proceed as if the events $L_{n-1}(v)$ were independent.
		We use this by invoking Chernoff's bound that says that if \(X\sim\operatorname{Bin}(m,q),\Theta\in(0,1)\), then $$\P(X<(1-\Theta)mq)\leq\exp\big\{-\tfrac{1}{2}\Theta^2mq\big\}.$$ 
		For \(q=\mathbb{P}(L_{n-1})\), \(m=D_n^d\) and \smash{\(\Theta=1-\frac{C_n}{D_n^d}\frac{1}{\mathbb{P}(L_{n-1})}\)} this leads to
		%\PM{Why do we have independence and therefore binomial variables?}
		%\MH{Problem: We want $\Theta=1-\frac{C_n}{D_n^d}\frac{1}{\mathbb{P}(L_{n-1})}$, but then $\Theta>1$ is possible!}
		\begin{align}\label{eq:enc}
			\P(E_n^c)&\leq\exp\Big\{-2^{d-1}\Big(1-\tfrac{p_B}{2^d}\tfrac{1}{\P(L_{n-1})}\Big)^2\P(L_{n-1})(n_*+n)^{2d}\Big\},\nonumber\\
			&=\exp\{-2^{-d-1}(2^d\P(L_{n-1})-p_B)^2(n_*+n)^{2d}\P(L_{n-1})^{-1}\}\nonumber\\
			&\leq \exp\{-2^{-d-1}(2^d\P(L_{n-1})-p_B)^2(n_*+n)^{2d}\},
		\end{align}
		%\MH{I'm still not happy with the bounds in \eqref{eq:enc}: there is a factor $2^{d}$ missing in the first line (which comes from $m=2^d(n+1)^{2d}$), and a factor $\P(L_{n-1})$ missing in the second: \begin{align}%\label{eq:enc}
				%	\P(E_n^c)&\leq\exp\Big\{-\frac{1}{2}\Big(1-\frac{1}{2^d}\frac{1}{\P(L_{n-1})}\Big)^2 \big(2(n+1)^2\big)^{d}\,\P(L_{n-1})\Big\},\nonumber\\
				%	&=\exp\{-2^{-d-1}(2^d\P(L_{n-1})-1)^2(n+1)^{2d}\,\P(L_{n-1})\}
				%\end{align}
				%The extra factor  $\P(L_{n-1})$  carries through the calculations below, but is dropping out eventually in the  $\P(L_{n}^c)$-calculation. 
				%}
			where we used the definitions of \(C_n\), \(D_n\) and the definition of \(E_n\) itself. This bound is only valid if $\P(L_{n-1})$ is sufficiently large for $\Theta$ to be smaller than $1$. For $n=2$, 
			this is satisfied by construction, since $\Theta$ simplifies to $1-\tfrac{1}{2^d}$. For $n\geq 3$ this follows inductively from the argument below.
			Combining (\ref{eq:enc}), (\ref{eq:fnc}) and (\ref{eq:gnc}) into (\ref{eq:lnc}), we obtain the recursive inequality
			\begin{align*}
				\P(L_n^c)&\leq\exp\{d\log(4)+4d\log(n_*+n)-\beta (n_*+n)^\varepsilon\wedge \tilde c n\log(n_*+n)\}\\
				&\quad+\exp\{-c_2p_B(n_*+n)^k\}+\exp\{-2^{-d-1}(2^d\P(L_{n-1})-p_B)^2(n_*+n)^{2d}\}.
			\end{align*}
			By setting $n_*$ large enough we get for $n\geq 2$ that
			\begin{align*}
				\P(L_n^c)&\leq 2\exp\{-c_3p_B\beta(n_*+n)^\varepsilon\wedge(n_*+n)^k\wedge n\log(n_*+n)\}\\&\quad+\exp\{-2^{-d-1}(2^d\P(L_{n-1})-p_B)^2(n_*+n)^{2d}\}.
			\end{align*}
			Define now the sequence $\ell_n:=\frac{1}{3}(n+1)^{-3/2}$  and observe that \(\sum_{i=1}^{\infty}\ell_i<1\). 
			%Next, note that by choosing \(n_*\) large enough, we have for all $\beta>0$ that \(\P(L_{1})\geq\ell_{1}\) by \eqref{eq:poisschernoff}. 
			We then obtain inductively for \(n \geq 2\) that
			\begin{align*}
				\P(L_n^c)&\leq 2\exp\{-c_3p_B\beta(n_*+n)^\varepsilon\wedge(n_*+n)^k\wedge (n_*+n)\log(n_*+n)\}\\&\quad+\exp\{-2^{-d-1}(2^d(1-\tfrac{1}{3}n^{-3/2})-p_B)^2(n_*+n+1)^{2d}\}\\
				&\leq\tfrac{1}{3}(n+1)^{-3/2}
				{=\ell_n,}
			\end{align*}
			where we finally take $n_*$ large enough for the second inequality to hold and use implicitly that $p_B$ approaches $1$ as $n_*$ is made large. This also yields that $\P(L_1^c)=1-p_B<\ell_1$.
			%We can now write 
			%\begin{align*}
			%	\P\Big(\bigcap_{n=1}^\infty L_n\Big)\geq 1-\sum_{n=1}^\infty\P(L_n^c)\geq 1-\sum_{n=1}^{\infty}\ell_i >0.
			%\end{align*}
			{Altogether, for $n^*$ large,} this gives the existence of the renormalized graph sequence with positive probability, and by Lemma \ref{prop:renorm} the result follows.
		\end{proof}
		
		This concludes the proof of Theorem \ref{thmRecurTrans} for the min kernel when \(\gamma>\frac{\delta}{\delta+1}\).
		As outlined at the beginning of the section, the preferential attachment and sum kernels dominate the min kernel and they are therefore, ceteris paribus, also transient.
		We now focus on the weight-dependent random connection model with the remaining kernels of Theorem \ref{thmRecurTrans}. The main difference is that these kernels lead to transient graphs by using direct connections between the dominant vertices, without having to rely on connector vertices. Therefore, we use the same strategy as for the min kernel but using direct connection probabilities instead of the two-connection probabilities from Lemma \ref{lemma:two_connection}. We first look at the product kernel, {for the corresponding lattice model, \emph{scale-free percolation}, the analogous result was obtained in \cite{HeydeHulshJorri17}.}
		
		\begin{proposition}\label{prop:product}
			{Consider the weight-dependent random connection model $\mathcal{G}^\beta$ with product kernel, where \(\gamma>\frac{1}{2}\) and \(\beta>0\). %i.e.\ the robust case. 
				Then $\mathcal{G}^\beta$ is transient almost surely.}
		\end{proposition}

		\begin{proof}
			%The following proof is an adaptation of the result in \cite{HeydeHulshJorri17}, adapted to the present framework. 
			%We again follow the strategy presented in the proof of \Cref{thrm:transience}. Since we can prove the claim using only direct connections between the vertices, Lemmas \ref{lemma:two_connection} and \ref{lemma:convolution} are no longer necessary. Furthermore, we use \Cref{lemma:rescale_single} instead of \Cref{lemma:rescale}.
			We claim that when \(\gamma>\frac{1}{2}\), the statement of Proposition \ref{prop:transience} holds for the product kernel as well. To see why, we repeat the steps of the proof of Proposition \ref{prop:transience}, setting \(C_n\) and \(D_n\) as before, replacing the value of \(u_n\) in \eqref{eq:un} by
			\[
			u_n=\beta^{\frac{1}{2\gamma}}d^{-\frac{d}{4\gamma}}(n_*+n+2)^{-\frac{\varepsilon}{2\gamma\delta}}2^{-\frac{d(n_*+n+2)}{2n_*\gamma}}(\tfrac{(n_*+n+3)!}{n_*!})^{-\frac{d}{\gamma}},
			\]
			with \(\varepsilon<d(2-\frac{1}{\gamma})\). Then, continuing along with the proof, we obtain instead of \eqref{eq:gnc} that
			\[
			\P(G_n^c\,|\,E_n)\leq\exp\{-c_2(n+1)^{d(2-\frac{1}{\gamma})}\},
			\]
			which is again a decreasing function of \(n\), since the exponent of \((n+1)\) is positive for \(\gamma>\frac{1}{2}\). The rest of the proof then proceeds unchanged.	
			%As argued in the proof of \Cref{corol:transience}, the renormalisation statement holds also for a lattice based version of the model (this is in fact the result of \cite{HeydeHulshJorri17}). Finally, using the same argument as in the proof of \Cref{thrm:transience} with \Cref{lemma:rescale_single} instead of \Cref{lemma:rescale} yields the result.
			%The result follows from \cite{HeydeHulshJorri17}, noting that all steps of the proofs stay as they were, with the exception of the initial step of the recursion in the proof of \cite[Proposition 5.3]{HeydeHulshJorri17}, which is lattice based. Therefore, we only need to modify the definition of good \(0\)-stage boxes. As in our proof of \Cref{prop:transience}, we define the points of the (thinned) Poisson point process to all be good \(0\)-stage boxes and proceed through the rest of the proof of \cite[Proposition 5.3]{HeydeHulshJorri17} without any further alterations. Then, our claim follows by applying this modified version of Proposition 5.3 from \cite{HeydeHulshJorri17} to prove the result for a sufficiently large value of \(\beta>\beta_c\) (analogous to  \Cref{prop:transience} in this paper) and then making the same coarse graining argument as in \cite{HeydeHulshJorri17} by using the unmodified version of Proposition 5.3 to prove the claim for general \(\beta>\beta_c\) (analogous to \Cref{thrm:transience}).
		\end{proof}
		
		The last kernel left to consider is the max kernel \(g^{\textrm{max}}\). Here, just like in the product kernel case, direct connections between vertices suffice to show that the graph is transient.
		
		\begin{proposition}
			{Consider the weight-dependent random connection model $\mathcal{G}^\beta$ with max kernel and any \(\gamma>0\). %i.e.\ the robust case. 
				For any $\beta>0$ we have that $\mathcal{G}^\beta$ is transient almost surely.}
		\end{proposition}
		
		\begin{proof}
			The result follows by repeating the proof of Proposition \ref{prop:product}, setting
			\[
			u_n:=\frac{1}{c_1}(n_*+n+2)^{-\frac{\varepsilon}{\delta(1+\gamma)}}2^{-\frac{(n_*+n+2)d}{n_*(1+\gamma)}}(\tfrac{(n_*+n+3)!}{n_*!})^{-\frac{2d}{1+\gamma}}
			\]
			where \(c_1\) is a constant depending on \(\rho\), \(d\) and \(\beta\), and \smash{\(\varepsilon<\frac{\gamma}{1+\gamma}\)}. Then, \(\P(G_n^c\,|\,E_n)\) is again decreasing in \(n\) precisely when \(\gamma>0\).
		\end{proof}
		\pagebreak[3]

		\subsection{Transience in the non-robust case}
		\label{sec:trannorob}
		\paragraph{The transience result.}
		We now turn to transience for non-robust supercritical percolation clusters in the weight-dependent random connection model. Of particular interest is the case $\lim_{r\to\infty}\rho(r)r^\delta=1$ with $\delta\in (1,2)$ and preferential attachment kernel $g=g^{\rm pa}$. However, we see in Theorem~\ref{thm:transnorob} that the precise form of $g$ is not very important, and in fact our argument applies for all the kernels discussed in Section~\ref{sec:intro}. The general strategy of the proof is similar to the robust case, in that we show that the infinite cluster contains a transient subgraph. \medskip
		
		On the one hand, in the non-robust regime (i.e., if $\beta_c>0$) this is a little more delicate, since the backbone of dominant vertices, which is only present in the robust case, cannot be used. On the other hand, the slow decay of connection probabilities as a function of distance allows us to relate the model to an instance of long-range percolation on the lattice via a coarse graining argument. A similar approach was used by Deprez and W\"uthrich \cite{DepreWuthr19}, the model there corresponds to our product kernel.
		We now reformulate the transience result for $\delta<2$.
		%in the non-robust situation (i.e., for $\gamma<\delta/(\delta+1)$ resp.\ $\gamma<1/2$) 
		%in the following claim: 
		\begin{theorem}\label{thm:transnorob}
			Let $\mathcal{G}^{\beta}$ denote the weight-dependent random connection model where
			\begin{enumerate}[(a)]
				\item $g$ is bounded, that is, $g^\ast:=\sup_{(s,t)\in(0,1)^2}g(s,t)<\infty$;\smallskip
				\item $\rho$ is regularly varying at $\infty$ of index $-\delta$ with $\delta\in(1,2)$. 
			\end{enumerate}
			Then, for all $\beta>\beta_c$, the graph $\mathcal{G}^{\beta}$ is transient almost surely.
		\end{theorem}\pagebreak[3]
		
		Mind that assumption (b) always holds if \eqref{eq:rhoass2} is satisfied and that assumption (a) is satisfied by all our kernels so that Theorem \ref{thm:transnorob} readily implies our transience claims if $\delta<2$. Let $B$ denote a non-empty domain in $\R^d$. A \emph{cluster in} $B$ is a connected component of the subgraph of $\mathcal{G}^\beta$ induced by all vertices with positions in $B$.
		The key result needed to establish transience is the following statement concerning percolation on finite boxes.
		
		\begin{proposition}[Local density of percolation clusters]\label{lem:perfin}
			Consider $\mathcal{G}^\beta$ with $\beta>\beta_c$ under the assumptions of Theorem~\ref{thm:transnorob}. For any $\lambda\in (0,1)$, and any $\varepsilon>0$, there exists an $M_0\in\N$, such that for all $M>M_0$
			\[
			\P\big(\text{there is a cluster of size at least }M^{\lambda d} \text{ in }[-M/2,M/2)^d\big)\geq 1-\varepsilon.
			\]
		\end{proposition}
		We postpone the proof of Proposition \ref{lem:perfin} and first show how to obtain Theorem~\ref{thm:transnorob} from Proposition~\ref{lem:perfin}. 
		\begin{proof}[Proof of Theorem \ref{thm:transnorob}]
			We apply a coarse-graining argument to relate the continuum model to a bond-site percolated lattice model based on the clusters obtained in Proposition~\ref{lem:perfin}. The coarse graining scheme has three parameters: a typically large $M\in\N$ controlling the coarse graining scale, a density exponent $\lambda\in(0,1)$ moderating the size of the coarse grained clusters and some small $\varepsilon\in(0,1)$ to account for the probability of local defects. As indicated by the notation, the relation of $M,\lambda$ and $\varepsilon$ is going to be dictated by Proposition~\ref{lem:perfin}. To construct the coarse grained lattice 
			configurations $\mathcal{G}'=\mathcal{G}'(M,\lambda,\varepsilon)$ from $\mathcal{G}^{\beta}$, consider the rescaled lattice $M\Zd$. To each site $v\in\Z^d$, we assign the box $B_v=Mv+[-M/2,M/2)^d$. A site $v$ is declared \emph{occupied} if 
			there is a cluster in $B_v$ with at least $M^{\lambda d}$ vertices. To every occupied site~$v$, we assign an \emph{occupying set} $W(v)$ of precisely $\lceil M^{\lambda d}\rceil$ vertices, which induce a connected subgraph of a 
			%\mathcal{G}^\beta$-
			cluster in $B_v$ as defined above. The existence of such vertex sets is guaranteed by the requirement that $v$ be occupied. Since there are typically many possible candidates for $W(v)$ for any given occupied lattice point $v$, we choose one in an arbitrary but \emph{local} fashion, i.e.\ the choice depends only on the marked point process configuration $\xi$ inside $B_v$. Two occupied lattice points $v,w\in\Z^d$ \emph{communicate} (denoted by $v\leftrightarrow w$) if their corresponding occupying sets are linked by an edge in $\mathcal{G}^{\beta}.$ For any two occupied lattice points $v,w\in\Z^d$ we have
			\begin{align*}
				\P(v \not\leftrightarrow w | \xi_{| B_v\cup B_w}) & \leq \P(\nexists (\x_1,\x_2)\in W(v)\times W(w): \{\x_1,\x_2\}\in E(\mathcal{G}^{\beta}))\\
				& \leq \prod_{((x_1,t_1),(x_2,t_2))\in W(v)\times W(w) }\big(1-\rho\big(g(t_1,t_2)|x_1-x_2|^d \big)\big),
			\end{align*}
			where $\xi_{| B_v\cup B_w}$ indicates the configuration restricted to the boxes $B_v,B_w$. Note that there exists some universal constant $C<\infty$, such that $|x_1-x_2|\leq C M k$ uniformly for all $(x_1,x_2)\in B_{v}\times B_{{w}}$, whenever $\|v-w\|_{\infty}=k\geq 1$. Moreover, since $g(t_1,t_2)\leq g^\ast<\infty$ by assumption and $|W(v)|=|W(w)|\geq {M^{d\lambda}}$ by construction for all occupied sites, we obtain the deterministic bound
			\begin{align*}
				\P(v \not\leftrightarrow w | v,w \text{ occupied}) 
				% & \leq 
				%\left(1-\rho\left(g(t_1,t_2)|x_1-x_2|^d \right)\right)^{M^{2d\lambda}}\\
				& \leq \exp\left(-M^{2d\lambda}\rho\big({g^\ast} (C M \|v-w\|_{\infty})^d \big)\right).
			\end{align*}
			Since $\rho$ is regularly varying of index $-\delta\in(-2,-1)$, we can find $z_0<\infty,c>0$ and $\delta^{\ast}<2$ such that $\rho(z)\geq c z^{-\delta^{\ast}}$ for all $z>z_0$. Consequently, we may fix $M_1$, such that for any $M>M_1$
			\begin{equation}\label{eq:lrpcomp}
				\P(v \leftrightarrow w | v,w\text{ occupied})\geq1 - \exp\left(-c_{\ast}\frac{M^{2d\lambda-\delta^{\ast}d}}{|v-w|^{\delta^{\ast} d}} \right)=:p'_{v,w},\quad v,w\in \Z^d,
			\end{equation}
			where $c_{\ast}=c_{\ast}(\rho,g^{\ast},d)$ is some constant independent of $\lambda.$ Note that the estimate in \eqref{eq:lrpcomp} holds uniformly for all configurations in which $v,w$ are occupied, independently of the occupation status of other vertices and whether they communicate or not. It thus follows from \eqref{eq:lrpcomp} that the configuration 
			$$\Big( \big\{\mathds{1}_{\{ v \text{ occupied} \}}: v\in \Z^d \big\}, \big\{\mathds{1}_{\{ v,w \text{ communicate} \}}: \{v,w\}\in (\Z^d)^{[2]}\big\} \Big)$$ 
			dominates a
			long-range percolation model on a site-percolated instance of $\Z^d$ with site retention probability $1-\varepsilon=\P(0 \text{ is occupied})$ and edge probabilities $p'_{v,w}$, for~$v,w\in \Z^d$. We denote this model by $\mathcal{G}'(M,\lambda,\varepsilon).$ By Proposition~\ref{lem:perfin}, we can choose $\lambda>\delta^{\ast}/2$ and then $M$ sufficiently large such that both $\varepsilon$ is arbitrarily close to $0$ and $p'_{v,w}$ dominates the connection probability in an arbitrarily dense long-range percolation cluster with edge decay parameter $\delta^{\ast}d<2d$. For such a choice of parameters it follows from \cite[Lemma 2.7]{Berge02} that $\mathcal{G}'$ is transient almost surely.
			\medskip
			
			It remains to argue why transience of the coarse grained model implies transience of the original graph. A slightly different argument in the same spirit is given in \cite[p.\ 545]{Berge02} for long-range percolation, but we provide an explicit construction. Let $V'$ and $E'$ denote the sites and edges of the coarse-grained configuration $\mathcal{G'}$. By Theorem \ref{thrm:effres}, transience of $\mathcal{G'}$ is equivalent to positivity of the {effective conductance} $C_{\textup{eff}}(0,\infty)$ in the electrical network obtained from $(V',E')$ in which each edge in $E'$ is assigned a unit conductance. %, see Section \ref{app:eclectic} of the appendix for a very brief introduction to the terminology.
			Recall that each node in $v\in V'\subset \Z^d$ is associated with a connected subgraph on the vertex set $W(v)$ in \smash{$\mathcal{G}_0^\beta$} contained inside the box $B_{v}$ and each edge in $E'$ incident to $v$ requires the existence of an edge in \smash{$\mathcal{G}_0^\beta$} with one endpoint $\x\in W(v)$. Choose for each edge such an endpoint and denote the sets of all the chosen points by~$X(v)$. Let $v\neq 0$, let $\mathcal{T}(v)$
			be a subtree of  \smash{$\mathcal{G}_0^\beta$} that spans $W(v)$ (which exists, because we required $W(v)$ to be connected) and equip the edges of $\mathcal{T}(v)$ with unit conductances. By monotonicity and the fact that  $W(v)$ contains at most $M^{\lambda d}+1$ vertices, the effective conductance between any two leaves of $\mathcal{T}(v)\subset\mathcal{G}'$ is bounded from below by the effective conductance between the two endpoints of a line graph with $M^{d}$ vertices. The same estimate holds for any two points in $X(v)$, since $X(v)$ is contained in $\mathcal{T}(v)$. It follows, again by monotonicity, that if we collapse \smash{$\mathcal{T}(v)\subset\mathcal{G}_0^\beta$} into the single vertex $v$ (whilst keeping all edges connecting $\mathcal{T}(v)$ to vertices in different boxes) and reduce the conductance on any edge $e$ incident to $v$ which corresponds to an edge in $\mathcal{G}'$ to $1/(1+M^{d})$, then we obtain a network with effective conductances which are dominated by those of the graph obtained from~$\mathcal{G}'$ by expanding the vertices into the respective trees $\mathcal{T}(v)$. An analogous argument also holds for the case $v=0$. We conclude, that if we assign to each edge of $\mathcal{G'}$ the conductance $(1+M^{d})^{-1}$ we obtain an electrical network whose effective conductance between any given vertex and $\infty$ are dominated by those of the subgraph of $\mathcal{G}^{\beta}$ obtained by expanding all $v\in V'$ into their corresponding $\mathcal{T}(v)$. This establishes transience of $\mathcal{G}^{\beta}$, since multiplying the conductances in~$\mathcal{G}'$ by a constant factor has the same effect on the effective conductance $C_{\textup{eff}}(0,\infty)$.
		\end{proof}
		%\medskip
		
		\begin{remark}
			For $\|v-w\|_{\infty}=1$, the estimate \eqref{eq:lrpcomp} yields that the coarse-grained random connection model dominates a (vertex-percolated) supercritical nearest-neighbour bond percolation cluster, which yields transience in $d\geq 3$ without recourse to long-range percolation, see~\cite{GrimmKesteZhang93}.
		\end{remark}
		%\smallskip
		
		\paragraph{Local density of percolation clusters.}
		We now turn to the proof of Proposition~\ref{lem:perfin}. We present a more direct (but weaker) sprinkling argument than the one given for the analogous (but stronger) statement \cite[Lemma 2.3]{Berge02} for long-range percolation. Fix $\beta>\beta_c$ and $\beta'\in(\beta_c,\beta).$ Our construction in Section \ref{secModel} provides a coupling $\mathcal{G}^{\beta}$ and $\mathcal{G}^{\beta'}$ from the same configuration $\xi$ in such a way that $\mathcal{G}^{\beta'}\subset\mathcal{G}^{\beta}$. The following lemma extends this observation. To formulate it, we write $q\mathcal{G}^{\beta}$ for the geometric graph obtained from $\mathcal{G}^{\beta}$ by performing independent Bernoulli percolation with retention probability $q\in[0,1]$ on the edges of $\mathcal{G}^{\beta}$.
		\begin{lemma}\label{lem:sprinkling}
			For any $\beta'\in(0,\beta)$ there exists some $q(\beta')\in(0,1)$ and a coupling of the random geometric graphs $\mathcal{G}^{\beta},\mathcal{G}^{\beta'}$ and $q(\beta')\mathcal{G}^{\beta}$ such that under the coupling
			$$E\big(\mathcal{G}^{\beta'}\big)\subset E\big(q(\beta')\mathcal{G}^{\beta}\big)\subset E\big(\mathcal{G}^{\beta}\big).$$
		\end{lemma}
		\begin{proof}
			The second inclusion is trivial for any $q(\beta')\in[0,1]$. The first inclusion follows by the observation that the scaling properties of the underlying Poisson process $\eta$ allow us to couple $\mathcal{G}^{\beta'}$ and $\mathcal{G}^{\beta}$ in such a way, that $\mathcal{G}^{\beta'}$ coincides precisely with Bernoulli percolation on the \emph{vertices} of $\mathcal{G}^{\beta}$ for some retention probability $q(\beta')\in(0,1)$. On any non-trivial connected graph, it is easily seen that Bernoulli percolation of edges and Bernoulli percolation of vertices with the same retention parameter $q$ can be coupled in such a way that the edge percolation configuration dominates the vertex percolation configuration. Applying this coupling to the connected components of $\mathcal{G}^{\beta}$ yields the first inclusion.
		\end{proof}\pagebreak[3]
		
		An important consequence of Lemma~\ref{lem:sprinkling} is that under the coupling
		\begin{equation}\label{eq:sprinkling}
			\P\big(\{\x,\y\}\in E(\mathcal{G}^{\beta}) 
			\mid\mathcal{G}^{\beta'} \big)\geq (1-q(\beta'))\, \rho({g^\ast} |x-y|^d),
		\end{equation}
		for all pairs $\x=(x,y),\y=(y,t)$ with $\x,\y\in \eta'$ independently of each other and independently of the occupation status of $\{\x,\y\}$ or any other edge in $\mathcal{G}^{\beta'}$. 
		\begin{proof}[Proof of Proposition~\ref{lem:perfin}]
			The estimate \eqref{eq:sprinkling} allows us to employ a simplified version\footnote{\label{percfindiscussion}
				%In principle, our proof is a simple way of fixing the faulty proof of \cite[Lemma 2.3]{Berge02} as it was originally published in 2002. The cost we pay is that we have to
				Our proof is simpler than the proof of \cite[Lemma 2.3]{Berge02} in the 2014 arXiv-version, but comes at the expense of \emph{assuming} $\beta>\beta_c$, whereas the variants of Proposition~\ref{lem:perfin} in \cite{Berge02} and \cite{DepreWuthr19} work under the assumption of the existence of an infinite component and \emph{imply} $\beta>\beta_c$. Since we only need the weaker statement, we give a self-contained proof of Proposition~\ref{lem:perfin} and remark in passing, that the stronger version can be obtained by adapting the proof of \cite[Lemma 6.1]{DepreWuthr19} to our model.} of the renormalisation	construction used in \cite{Berge02}. The basic building blocks of the renormalisation scheme are localised clusters in $\mathcal{G}^{\beta'}$ which we construct now. Fix $\beta'\in(\beta_c,\beta)$ and let $\mathcal{C}(\beta')$ denote the infinite cluster in $\mathcal{G}^{\beta'}$. Consider the collections of boxes $$\mathfrak{B}(l)=\{ lx+[-l/2,l/2)^d \colon x\in \Zd\}, \quad l\in\mathbb{N}.$$ Let $m\in\mathbb{N}$ be given. For any box $B\in \mathfrak{B}(m)$, we write $\mathcal{V}(B)$ for those vertices of $\mathcal{C}(\beta')$ that have positions in $B$. 
			Let $\varepsilon>0$ and $\theta_0\in(0,\theta(\beta'))$. Then we may choose $m$ so large that 
			$$
			\P(|\mathcal{V}(B)|\geq \theta_0 m^d)\geq 1-\varepsilon/2\qquad\text{for all $B\in\mathfrak{B}(m)$}
			$$
			since the probability of the left hand side converges to $1$ as $m\to\infty$ due to ergodicity.
			For $B\in \mathfrak{B}(m)$, we now define the event  
			$$E_0(B)=\{ |\mathcal{V}(B)|\geq \theta_0 m^d, \mathcal{V}(B) \text{ is contained in a connected component of }\mathcal{G}^{\beta'}\cap \partial_k(B) \},$$
			where $\partial_k(B)$ denotes the $k$-neighbourhood of $B$ with respect to $\|\cdot\|_{\infty}$ in $\R^d$ and $\mathcal{G}\cap A$ is a shorthand for the subgraph of $\mathcal{G}$ induced by vertices with positions in $A\subset \R^d.$
			Since the infinite cluster is unique, there exists $k=k(m)\in\N$ such that for any $B\in\mathfrak{B}(m),$
			$$
			\P(E_0(B))>1-\varepsilon.
			$$
			If $E_0(B)$ occurs for some $B\in\mathfrak{B}(m)$, then we call the associated collection $\mathcal{V}(B)$ of vertices a $(\theta,m,k)$-\emph{precluster} (in $B$).\smallskip
			
			Our next goal is to recursively construct more and more strongly localised clusters from the preclusters in $\mathcal{G}^{\beta}$ using the independent sprinkling induced via \eqref{eq:sprinkling}. By our standing assumption on the regular variation of $\rho$, we may further increase $m$ (and thus possibly $k(m)$) such that $\rho(v)\geq c_\rho v^{-\delta^{\ast}}$ for all $v\geq {g^{\ast}}d^{d/2}m^d$ and some $\delta^{\ast}\in(\delta,2).$ Let $\lambda\in(0,1)$ be arbitrary and choose 
			$$
			\eta>\frac{2}{d}\left(\frac{1}{1-\lambda} \vee \frac{1}{1-\delta^\ast/2} \right).
			$$
			Now set
			$$\theta_n=\frac{1}{n^2}, \text{ and } \sigma_n=\lceil(n+1)^{\eta}\rceil,\quad n\in\mathbb{N}.$$
			The cubes $$B\in \mathfrak{B}\left(m \prod_{i=1}^n \sigma_i \right),$$
			are called \emph{stage}-$n$ boxes. Note that for $n\geq 1$, each stage-$n$ box $B$ contains precisely $\sigma^d_n$ stage-$(n-1)$ boxes, which we call the \emph{subboxes} of $B$. A stage-$0$ box $B$ is \emph{good} if the event $E_0(B)$ occurs, i.e.\ if it contains a $(\theta_0,m,k)$-precluster. A \emph{stage-$1$ precluster} inside a stage-$1$ box $B$ is a collection of $\theta_1\sigma_1$ preclusters, each contained in a different subbox of $B$, which are all at graph distance $1$ of each other in $\mathcal{G}^{\beta}$. Similarly, for $n>1$,
			a \emph{stage-$n$ precluster} inside a stage-$n$ box $B$ is a collection of $\theta_n\sigma_n$ stage-$(n-1)$ preclusters, each contained in a different subbox of $B$, which are all at graph distance $1$ of each other in $\mathcal{G}^{\beta}.$ If $n\geq 1$, then a stage-$n$ box is \emph{good} if it contains a stage-$n$ precluster.\smallskip
			
			Note that a stage-$n$ precluster contains at least
			$$
			v_n:=\theta_0m^d \prod_{i=1}^n \theta_n\sigma_n
			$$
			vertices. Consequently, if a stage-$n$ box $B$ is good, then there must be a connected component in $\mathcal{G}^{\beta}\cap \partial_k(\Gamma)$ of size at least $v_n$. For any stage-$n$ box $B$, let $\tilde B$ denote the union of $B$ with all its $3^d-1$-neighbouring stage-$n$ boxes (including diagonal neighbours). Since $k$ depends only on $m$, but not on $n$ it follows that there exists a stage $N$ such that for all $n\geq N$ and any stage-$n$ box $B$, we have
			$\partial_k(B)\subset \tilde{B}$ and we conclude that on the event
			$\{B \text{ is good} \}$ there exists a connected component of $\mathcal{G}^\beta$ inside $\tilde{B}$ of size at least
			$$
			v_n=\theta_0m^d \prod_{l=1}^n \theta_l\sigma_l^d\geq q_1\theta_0 m^d \prod_{l=1}^n l^{-2+ \eta d} = q_1\theta_0 m^d\left(\prod_{l=1}^n l^{\eta d}\right)^{1-\frac{2}{\eta d}}\geq q_2 \frac{\theta_0}{3^{\lambda d}}\text{vol}(\tilde B)^\lambda
			$$
			for some constants $q_1,q_2$, which implies the claim of the proposition once we show that the probability $\P(B \text{ is not good})$ can be made arbitrarily small for sufficiently large $n$.\medskip
			
			To this end, we consider for each $n$ the $n$-stage box $B_n$ containing the origin, which is sufficient due to translation invariance. Let $E_n$ denote the event that $B_n$ is good and let $F_n$ denote the event that at least $\theta_n \sigma_n$ subboxes of $B_n$ are good. It follows from Markov's inequality that
			\begin{equation}\label{eq:Markov}
				\P(F_n^\texttt{c})=\P\bigg( \sum_{B'\subset B_n \text{ subbox}}\mathds{1}\{ B' \text{ is not good}  \}> (1-\theta_n)\sigma_n^d  \bigg) \leq \frac{1}{1-\theta_n}\P(E_{n-1}^\texttt{c}),\quad n\in\mathbb{N}.
			\end{equation}
			Let now $B^{(1)},B^{(2)}$ be good subboxes of $B_n$, $n\geq 1$, and let $\mathcal{V}_1,\mathcal{V}_2$ denote their largest stage-$(n-1)$ preclusters. Note that these can be determined by considering only $\mathcal{G}^{\beta'}$ and additional edges in $\mathcal{G}^{\beta}$ within the subboxes $B^{(1)},B^{(2)}$. By observing that the estimate \eqref{eq:sprinkling} holds independently for each potential edge between $\mathcal{V}_1,\mathcal{V}_2$, we obtain that, given $\mathcal{G}^{\beta'}$ and all box-to-box connections made in previous stages,
			\begin{equation}\label{eq:firstestimate}
				\begin{aligned}
					\P(\mathcal{V}_1,\mathcal{V}_2 \text{ are not at distance }1 \text{ in }\mathcal{G}^\beta) &\leq\prod_{\substack{((x_1,t_1),(x_2,t_2))\\ \in\mathcal{V}_1\times\mathcal{V}_2}}\big(1-(1-q(\beta'))\rho(g_\ast |x_1-x_2|^d)\big)\\
					&\leq \prod_{((x_1,t_1),(x_2,t_2))\in\mathcal{V}_1\times\mathcal{V}_2}\textup{e}^{-(1-q(\beta'))\rho(g_\ast |x_1-x_2|^d)}\\
					&\leq \prod_{((x_1,t_1),(x_2,t_2))\in\mathcal{V}_1\times\mathcal{V}_2}\textup{e}^{-(1-q(\beta')) \rho(g_\ast\,\textup{diam}(B_n)^d )}\\
					& \leq \textup{e}^{-(1-q(\beta')) \rho(g_\ast\,\textup{diam}(B_n)^d )v_{n-1}^2}.
				\end{aligned}
			\end{equation}
			Now note that $\textup{diam}(B_n)^d= d^{d/2}\textup{vol}(B_n)$, and that $\rho(g_\ast d^{d/2}\,\textup{vol}(B_n)) \geq c \textup{vol}(B_n)^{-\delta^{\ast}}$ by choice of $m$ for some constant $c$. We thus conclude from \eqref{eq:firstestimate} that there exists some small constant $\nu=\nu(d,\beta',g,\rho)>0$ such that
			\begin{equation}\label{eq:secondestimate}
				\P(\mathcal{V}_1,\mathcal{V}_2 \text{ are not at distance }1 \text{ in }\mathcal{G}^\beta)\leq \textup{e}^{-\nu\, v_{n-1}^2 \textup{vol}(\Gamma_n)^{-\delta^\ast}},
			\end{equation}
			and that this estimate holds independently for any pair of good subboxes $B^{(1)},B^{(2)}$ and independently of the formation of all previous stages of the recursion.\medskip

			Combining \eqref{eq:secondestimate} with a simple union bound, we obtain that
			\begin{equation}\label{eq:E_n_est}
				1-\P(E_n|F_n)\leq \sigma_n^{2d}\textup{e}^{-\nu\, v_{n-1}^2 \textup{vol}(B_n)^{-\delta^\ast}}.
			\end{equation}
			Since 
			\begin{align*}v_{n-1}^2 \textup{vol}(B_n)^{-\delta^\ast} & = \sigma_n^{-\delta^\ast d}(\theta_0m^d)^{2-\delta^\ast}\prod_{l=1}^{n-1}\theta_l^2\prod_{l=1}^{n-1}\sigma_l^{2d-\delta^\ast d}
			\\	& \geq c\, \theta_0 m^{2d-\delta^\ast d} n^{-\delta^\ast d}((n-1)!)^{\eta(2-\delta^\ast)d-4}
			\end{align*}
			for some constant $c$, it is straightforward (if necessary by further increasing $m$) to deduce the existence of small constants $\nu_1,\nu_2$ such that for all $n\geq 1$
			$$ 
			1-\P(E_n|F_n)\leq \textup{e}^{-\nu_1 ((n-1)!)^{\nu_2}}\leq \varepsilon \theta_n,
			$$
			In particular, combining this estimate with \eqref{eq:Markov} we have that
			$$\P(E_n^\texttt{c})\leq \P(F_n^\texttt{c})+\P(E_n^\texttt{c}|F_n)\leq \varepsilon \theta_n + \frac{1}{1-\theta_n} \P(E_{n-1}^\texttt{c})\leq \varepsilon \theta_n + (1+2\theta_n) \P(E_{n-1}^\texttt{c}),\quad \mbox{ for } n\geq 1.$$
			It follows that
			$$
			\P(E_n^\texttt{c})\leq (1+3\theta_n)(\varepsilon \vee \P(E_{n-1}^{\texttt{c}})),\quad \mbox{ for } n\geq 1,
			$$
			and recursively we obtain that 
			\begin{equation}\label{eq:recursioneq}
				\P(E_n^\texttt{c})\leq 2\varepsilon\prod_{l=1}^{n}(1+3\theta_l), \quad \mbox{ for }n\in\mathbb{N}.
			\end{equation}
			Since $(\theta_n)$ is summable, it follows that the right hand side of \eqref{eq:recursioneq} is uniformly bounded by some constant multiple of $\varepsilon$ which suffices to conclude the proof, since $\varepsilon$ was chosen arbitrarily.
			%We have
			%\[
			%\P(|C'_n|\geq s_*\theta' n^d) = \P\left(\sum_{(x,s):x \in B_n}\mathds{1}\{(x,s)\in  C_\infty, s\leq s_*\} \geq s_*\theta' n^d\right).
			%\]
			%By ergodicity and standard results on marked point processes (see e.g.\ \cite[Chapter 13.4]{DaleyVereJ08}) it holds true that, in probability,
			%\[
			%s_*^{-1}n^{-d}\sum_{(x,s)\in B_n: s\leq s_*}\mathds{1}\{(x,s)\in  C_\infty\} \overset{n\to\infty}{\longrightarrow} \theta.
			%\]
			%It follows that we can find $n_1$, such that for all $n\geq n_1$, we have 
			%\[
			%\P( n^{-d}|C'_n|\geq s_*\theta')\geq 1-\nicefrac{\varepsilon}{2}.
			%\]
			%Fix such $n$ and observe that by uniqueness of the infinite cluster, we must have that $C'_n$ is part of the same cluster within $B_n^K$ for some random $K<\infty$. Consequently, we can find $k(n)$, such that 
			%\[
			%\P(C'_n \text{ is not connected in } B^{k(n)}_n)< \nicefrac{\varepsilon}{2},
			%\]
			% and the statement follows by taking a union bound.
		\end{proof}

					\section{Recurrence}\label{secRecurrence}
					In this section, we prove the recurrence results of Theorem \ref{thmRecurTrans}. We develop continuum versions of sufficient criteria for recurrence from \cite{Berge02}. As a rule of thumb, these criteria can only work if the correlations between edges induced by the vertex marks are not too strong. In particular it is necessary that the expected number of neighbours of $0$ has a finite second moment, therefore all  recurrence conditions in Theorem~\ref{thmRecurTrans} require (at least) that $\gamma<1/2$. We work with general geometric random graphs with vertices given by an (unmarked) unit intensity Poisson process. Therefore, to keep the notation concise when we specialise to our model, we break with our conventional notation and identify vertices $\x=(x,s)$ with their location $x$ throughout this section. Consequently, we also view potential edges as elements of \smash{$\eta^{{[2]}}$} or $\eta_0^{_{[2]}}$ instead of the corresponding marked sets. 

						\begin{lemma}\label{thrm:recurrence2}
							Let $\X_\infty$ be a unit intensity Poisson process on \(\R^2\). Consider a random graph $\mathcal{H}$ on this point process, where points $x,y\in \X_\infty=V(\mathcal{H})$ are joined by an edge with conditional probability \(P_{|x-y|}\), given $\X_\infty$. If \[\limsup_{r\to\infty}r^\alpha P_{r}<\infty,\]
							for some $\alpha>4$, then any infinite component of $\mathcal{H}$ is recurrent.
						\end{lemma}
					
Lemma \ref{thrm:recurrence2} is a continuum adaptation of a result by Berger \cite{Berge02} for general lattice-based long-range percolation models\footnote{{Note that like Berger's version \cite[Theorem 3.10]{Berge02}, the result remains true if we only assume translation invariance. However, we restrict ourselves to connection probabilities that depend only on distance, since it saves some notation. Furthermore, \cite[Theorem 3.10]{Berge02} includes the case $\alpha=4$. After thorough inspection of the arguments in \cite{Berge02}, we believe that this is only justified in the case of percolation models with \emph{uncorrelated} edges. Otherwise, one needs some control on the correlations, cf. our recurrence proof for the preferential attachment kernel below.}}. Note that, in particular, no assumptions on the (in)dependence of the edges are required. Before we prove Lemma~\ref{thrm:recurrence2} below, we discuss its consequences for the kernels we are interested in. 
					
\begin{proof}[{Proof of Theorem~\ref{thmRecurTrans}, recurrence for min, sum and product kernel if $d=2$}]  We wish to apply Lemma~\ref{thrm:recurrence2}. We have $g^{\textup{sum}}\leq g^{\textup{min}}$ and we are only interested in upper bounds on the connection probabilities, thus we only need to treat the sum and product kernels. Recall that by \eqref{eq:rhoass2} we may bound $\rho(z)\leq c_\rho z^{-\delta}, z>0$. For $S,T$ independent on $(0,1)$ uniformly distributed random variables we thus find, for $r>1$,
						\begin{equation}\label{eq:PrBound1}\begin{aligned}
							P_r & = 2\E\big[\tfrac12\rho(g(S,T)r^{-2})\big] \leq 2\left(1-\E\left[\textup{e}^{-\rho(g(S,T)r^{-2})}\right]\right)\\
							& \leq 2\left(1-\E\left[\textup{e}^{-g(S,T)^{-\delta}{c_\rho r^{-2\delta}}}\right]\right).\end{aligned}
						\end{equation}
						The expectation on the right hand side is the Laplace transform $\hat{F}_{Z(g)}$ of $Z(g)=g(S,T)^{-\delta}$ evaluated at 
						\smash{$\lambda={c_\rho r^{-2\delta}}.$} Now observe that, for $x>0$,
						\begin{equation}\label{eq:kernelbd}
							\P\big(Z(g)>x \big)= \begin{cases}
								\P\big( S^{-\frac \gamma2}+ T^{-\frac\gamma2}>{\beta^{-\frac12} x^{\frac{1}{2\delta}}}\big), &\text{ if } g=g^{\textup{sum}},\\
								\P\big( S^{-\gamma}T^{-\gamma}>{\beta^{-\frac12}  x^{\frac{1}{\delta}}}\big), &\text{ if } g=g^{\textup{prod}}.
							\end{cases}
						\end{equation}
						Now the right hand side of \eqref{eq:kernelbd} is the tail probability of either the sum or the product of two independent power-law random variables and it is easy to see that the tail index $\alpha\in(0,\infty)$ of a heavy-tailed distribution remains unchanged under independent multiplication or summation, respectively. Evaluating the corresponding tail indices, we now infer from \eqref{eq:kernelbd} that there exists some slowly varying function $\ell$ (depending on $g$), such that 
						\begin{equation}\label{eq:tailvariation}
							\P\big(Z(g)>x \big)= \ell(x) x^{-\frac{1}{\gamma\delta}}\quad \mbox{ for } x>0.
						\end{equation}
						{By monotonicity {of $P_r$ in $\gamma$}  we may assume that $\gamma>1/\delta$} (recall that $1/\delta<1/2$ by assumption) and in this case it is a consequence of Karamata's Tauberian Theorem \cite[Corollary 8.1.7]{BinghGoldiTeuge87} that \eqref{eq:tailvariation} is equivalent to the existence of some $\varepsilon_g>0$ and some slowly varying function $\tilde{\ell}$ such that
						\begin{equation*}
							{	1-\hat{F}_{Z(g)}(\lambda)} = \tilde{\ell}(\lambda) \lambda^{\frac{1}{\gamma\delta}}\quad \text{for }\lambda\in(0,\varepsilon_g).
						\end{equation*}
						Inserting this asymptotic bound into \eqref{eq:PrBound1} yields that there exits a third slowly varying function $\bar{\ell}$ such that
						\smash{$P_r \leq \bar{\ell}(r)r^{-{2}/{\gamma}}$} for all sufficiently large $r$,
						which, by Lemma~\ref{thrm:recurrence2}, implies recurrence of the corresponding graphs whenever $\gamma<1/2$.
					\end{proof}
						In order to apply the Nash-Williams criterion (Theorem~\ref{thrm:nash1}) and prove Lemma \ref{thrm:recurrence2}, we adapt the the approach from \cite{Berge02} to our setting. To begin with, we turn the graph into an electrical network by assigning conductance $1$ to each edge.
						
						\medskip
						
						\noindent\emph{Discretisation:} We assign to each point $x=(x_1,x_2)\in\Z^2$ the half-open cube $\Gamma_x=[x_1-1/2,x_1+1/2)\times [x_2-1/2,x_2+1/2)$. For each $x$, we collapse all Poisson points inside $\Gamma_x$ into $x$, remove any resulting loops and retain all other edges with their conductances. This yields a long range electrical network with parallel edges that dominates the original graph in terms of effective conductance. Now, if a point pair $\{x,y\}$ is joined by $m>1$ parallel edges, we remove $m-1$ of them and increase the conductance on the remaining edge to $m$. The parallel law for electrical networks implies that the effective conductance of the electrical network is not affected by this step. We denote the resulting network  on $\Z^d$ by $\mathcal{L}=\mathcal{L}(\mathcal{H})$. Note that
						\[
						\P(\{x,y\}\in E(\mathcal{L}))\leq \P(\exists \{u,v\}\in E(\mathcal{H}): u\in \Gamma_x, v\in \Gamma_y)\leq \sup_{x^\ast\in \Gamma_x, y^\ast\in \Gamma_y} P_{|x^{\ast}-y^{\ast}|},
						\]
						where we have used that $\E|\X_\infty\cap\Gamma_x|=1$. Since the diameter of the cubes $K_x,x\in \Z^2$ is uniformly bounded, it follows that if    
						\[\limsup_{r\to\infty}r^\alpha P_{r}<\infty\]
						for some $\alpha>0$, the same must be true for the connection probabilities of lattice points in $\mathcal{L}$.
						
						\medskip
						
						\noindent\emph{Edge projection:} Given the lattice-based long-range model $\mathcal{L}$, we now project the long edges to nearest-neighbour edges: For every long (i.e.\ not nearest-neighbour) edge $\{x,y\}=\{y,x\}$ in $\mathcal{L}$ with conductance \(c_{xy}\), we remove the edge $\{x,y\}$ and increase the conductance for every nearest-neighbour bond in the shortest rectangular nearest-neighbour cycle containing $x$ and $y$  by \(c_{xy}(|x_1-y_1|+|x_2-y_2|)\). Due to the serial and parallel laws, the effective conductance is not decreased by this step. The resulting nearest-neighbour network is denoted by $\mathcal{N}$.
						
					\begin{lemma}\label{lemma:projected}
							Let $\X_\infty$ be a unit intensity Poisson process on \(\R^2\). Consider a random graph on this point process, where points $x,y\in \X_\infty$ are connected with probability \(P_{|x-y|}\), 
								%where \(P_\cdot\colon \R_+\to[0,1]\), \(v\mapsto P_v\)  be 
								such that \[\limsup_{r\to\infty}r^\alpha P_r<\infty,\]
								for some $\alpha>3$.
							For the electrical network \(\mathcal{N}\) with conductances $(C_e)_{e\in E(\Z^2)}$ constructed above, we have that:
							\begin{enumerate}[(1)]
								\item All conductances are finite almost surely.
								\item The effective conductance of the network is bigger or equal to the effective conductance of the original network.
								\item The distribution of edge conductances in \(\mathcal{N}\) is shift invariant.
								\item If $\alpha>4$, then the conductance \(C_e\) of an edge has a Cauchy tail, i.e.\ there exists a constant \(c\) such that \(\P(C_e>cn)\leq n^{-1}\) for every \(n\geq 1\).
							\end{enumerate}
						\end{lemma}
						To prove the lemma, it is convenient to first introduce some notation. For $\ell\in \N$ and any nearest-neighbour edge $e\in\E(\Z^2)$, let $\Pi_\ell(e)$ denote the set of all potential edges in $\X_\infty^{[2]}$ which contribute an amount of $\ell$ to $C_e$, i.e.\ pairs $\{x,y\}\in \X_\infty^{[2]}$ which are discretised to $(x'_1,x'_2),(y'_1,y'_2)$ with $|x'_1-x'_2|+|y'_1-y'_2|=\ell$. Since the discretisation step shifts points by a uniformly bounded distance, it follows that there exist  global constants $d_1,d_2$, such that
						\begin{equation}\label{eq:distanceconversion}
						d_1|x-y|\leq |x'_1-x'_2|+|y'_1-y'_2|\leq (d_2|x-y|) \vee 1
						\end{equation}
						for any $\{x,y\}\in\X_\infty$, and since $\X_\infty$ is a homogeneous Poisson point process we have that
						\begin{equation}\label{eq:exppicounts}
							\E |\Pi_\ell(e)|\asymp \ell^2\quad \text{ and }\quad \E |\Pi_\ell(e)|^2\lesssim \ell^4,
						\end{equation}
						as $\ell\to\infty$ for any $e\in\E(\Z^d)$.
						 
						\begin{proof}[Proof of Lemma~\ref{lemma:projected}]
							Our construction ensures that assertions (2) and (3) are satisfied. Let us now show (1). By (3), it suffices to consider the fixed edge $e_0=\{(0,0),(1,0)\}$. Clearly, it holds that
							$$C_{e_0}=\sum_{\ell\in\N}\ell \sum_{\{x,y\}\in \Pi_\ell(e_0)}\mathds{1}\{\{x,y\}\in E(\mathcal{H})\}$$
							By \eqref{eq:distanceconversion} and \eqref{eq:exppicounts}, we have that the probability that the $\ell$-th term in the outer sum is non-zero is at most of order
							\[
							\ell^2 P_{\ell/d_2}\lesssim \ell^{2-\alpha}.
							\]
							Consequently, $C_{e_0}$ is almost surely finite if $\alpha>3$ by an application of the Borel-Cantelli Lemma. Finally, to show (4), we apply Markov's inequality and obtain
							\[
							\P(C_{e_0}>n)\leq \frac{\E C_{e_0}}{n},
							\]
							which implies (4), since
							\[
							\E C_{e_0}\lesssim \sum_{\ell=1}^\infty \ell^3 P_{\ell/d_2} \lesssim \ell^{3-\alpha}< \infty
							\]
							if $\alpha>4$, where we have again used \eqref{eq:distanceconversion} and \eqref{eq:exppicounts}.
								\end{proof}

							The key to establishing Lemma \ref{thrm:recurrence2} is one final result from \cite{Berge02}, which we cite without~proof.
							\begin{theorem}[{\cite[Theorem 3.9]{Berge02}}]\label{thrm:cauchy}
								Let \(G\) be a random electrical network on the lattice \(\Z^2\), such that all of the edges have the same conductance distribution, and this distribution has a Cauchy tail. Then \(G\) is almost surely recurrent.
							\end{theorem}
							Notice that we do not require any independence in the theorem.
							\begin{proof}[{Proof of Lemma \ref{thrm:recurrence2}}]
								By the steps (A)-(E) and the remarks following each step, the conductance between two lattice points of the projected electrical network is bigger than the effective conductance  between their preimages in the original graph in the continuum. Therefore, if the projected electrical network is recurrent, so is the original graph. By Lemma \ref{lemma:projected}, the conductances of the projected network have Cauchy tails and by Theorem \ref{thrm:cauchy}, this implies that the network is recurrent.
							\end{proof}
						
\paragraph{Recurrence result for preferential attachment kernel in $d=2$.}							We wish to apply Theorem~\ref{thrm:cauchy}, but we unfortunately cannot rely on Lemma~\ref{thrm:recurrence2}, since even for $\gamma<1/2$ the preferential attachment kernel induces a model with $P_{x,y}\asymp |x-y|^{-4}$, as we will see below. To address this complication, we need to supplement first moment bounds on the edge probability by bounds on edge correlations. To formulate bounds on connection probabilities in a concise manner, we rely on translation invariance and consider, as an auxiliary object, the unit intensity Poisson process $\eta_{0,r}$ on $\R^2$ with the two additional points $(0,0)$ and $(r,0)$, where $r>0$. Throughout the remainder of the proof, we always keep $\rho(z)\asymp z^{-\delta}$ with $\delta>2$ and an instance of the preferential attachment kernel $g^{\textup{pa}}$ with $\gamma<1/2$ fixed. 
	\begin{lemma}\label{lem:connprobcorr}
	For $r>0$ and $s\in(0,1]$ write 
		\[Q_r(s)=\int_0^1\rho(\beta^{-1} g^{\textup{pa}}(s,u)r^2)\,\textup{d}u\]
	for the probability that $((0,0),s)$ is connected with a given point at distance $r$.
	If $\gamma<1/2$ and $\delta>2$, then we have for $s\in(0,1)$ and $r>1$,
								\[
								Q_r(s)\lesssim \begin{cases} \mathds{1}\Big\{s<r^{-\frac{2}{\gamma}}\Big\}+\mathds{1}\Big\{r^{-\frac{2}{\gamma}}\leq s < r^{-2}\Big\}\,r^{-\frac{2}{1-\gamma}}s^{-\frac{\gamma}{1-\gamma}}+\mathds1\big\{r^{-2}\leq s\big\} \,r^{-\frac2\gamma}s^{-\frac{1-\gamma}\gamma}\\ \text{\hspace{10.2cm} if }\gamma>1/\delta,\\
									\mathds{1}\Big\{s<r^{-\frac{2}{\gamma}}\Big\}+\mathds{1}\Big\{r^{-\frac{2}{\gamma}}\leq s < r^{-2}\Big\}\,r^{-\frac{2}{1-\gamma}}s^{-\frac{\gamma}{1-\gamma}}+\mathds1\big\{r^{-2}\leq s\big\} \,r^{-2\delta}s^{1-\delta} \\ \text{\hspace{10.2cm} if }\gamma<1/\delta.
								\end{cases}\]
							\end{lemma}
							Here and in the remainder of this section, we use the notation $f(x)\lesssim g(x)$, if there is a constant $C<\infty$ (which may depend on $\beta,\gamma,\delta$, the function \(\rho\) and the kernel $g^{\textup{pa}}$, but not on $r$ and $s$) such that $f(x)\leq Cg(x)$ for all sufficiently large $x$.
							\begin{proof}[Proof of Lemma~\ref{lem:connprobcorr}]
								We first assume that $\gamma\in(1/\delta,1/2)\neq\emptyset$ (recall that $\delta>2$). Due to our assumption on $\rho$ there exists a constant \(\tilde{c}\in (1,\infty)\) depending on the choice of the function \(\rho\), 
								the kernel $g^{\textup{pa}}$ and $\beta$,
								such that the connection probability (for {\(0<s\le t<1\)}) is bounded from above by
								\[
								\rho(g^{\textup{pa}}(s,t)r^2)\leq \tilde{c} \big(1\wedge (s^{\gamma}t^{(1-\gamma)}r^2)^{-\delta}\big).
								\]
								Integrating over the mark of $(r,0)$ with $r>1$ fixed, we obtain that
								\[
								Q_r(s)\lesssim \int_{0}^s 1\wedge (u^{\gamma}s^{(1-\gamma)}r^2)^{-\delta}\,\textup{d}u + \int_{s}^1 1\wedge (s^{\gamma}u^{(1-\gamma)}r^2)^{-\delta}\,\textup{d}u =: I_1(s) + I_2(s). 
								\]
								Evaluating the integrals yields
								\begin{align*}
									I_1(s) & \lesssim \mathds{1}\{s< r^{-2} \}\,s + \mathds{1}\{r^{-2} \leq s \}\Big(r^{-\frac{2}{\gamma}}s^{-\frac{1-\gamma}{\gamma}} + r^{-2\delta-\frac{2}{\gamma}(1-\gamma\delta)}s^{-(1-\gamma\delta)-\frac{1-\gamma}{\gamma}(1-\gamma\delta)} \Big)\\
									& \lesssim\mathds{1}\{s< r^{-2} \}\,s+\mathds{1}\{r^{-2} \leq s \}\,r^{-\frac{2}{\gamma}}s^{-\frac{1-\gamma}{\gamma}},
								\end{align*}
								and
								\begin{align*}
									I_2(s) & \lesssim \mathds{1}\Big\{s< r^{-\frac{2}{\gamma}} \Big\}(1-s) \\
									& \phantom{\lesssim} + \mathds{1}\Big\{r^{-\frac{2}{\gamma}} \leq s < r^{-2} \Big\}\Big(r^{-\frac{2}{1-\gamma}}s^{-\frac{\gamma}{1-\gamma}}+r^{-2\delta-\frac{2}{1-\gamma}(1-(1-\gamma)\delta)}s^{-\gamma\delta-\frac{\gamma}{1-\gamma}(1-(1-\gamma)\delta)}\Big)\\
									& \phantom{\lesssim} + \mathds{1}\{r^{-2}\leq s \}\,r^{-2\delta}s^{1-\delta}\\
									&\lesssim \mathds{1}\Big\{s< r^{-\frac{2}{\gamma}}\Big\} + \mathds{1}\Big\{r^{-\frac{2}{\gamma}} \leq s < r^{-2} \Big\}\,r^{-\frac{2}{1-\gamma}}s^{-\frac{\gamma}{1-\gamma}}+\mathds{1}\{r^{-2}\leq s \}\,r^{-2\delta}s^{1-\delta},
								\end{align*}
								where we have used for $I_1$ that $\gamma>1/\delta$ and for $I_2$ that $\gamma<1/2<(\delta-1)/\delta.$ Combining both expressions and using that both $r^{-2/\gamma}s^{-(1-\gamma)/\gamma}\geq r^{-2\delta}s^{1-\delta}$ (where $\gamma>1/\delta$) precisely if $s\geq r^{-2}$,  as well as $r^{-2/(1-\gamma)}s^{-\gamma/(1-\gamma)}> s$ precisely if $s<r^{-2}$, we obtain that
								\[
								Q_r(s)\lesssim \mathds{1}\Big\{s<r^{-\frac{2}{\gamma}}\Big\}+\mathds{1}\Big\{r^{-\frac{2}{\gamma}}\leq s < r^{-2}\Big\}\,r^{-\frac{2}{1-\gamma}}s^{-\frac{\gamma}{1-\gamma}}+\mathds1\big\{r^{-2}\leq s\big\} \,r^{-\frac2\gamma}s^{-\frac{1-\gamma}\gamma},
								\] 
								as claimed. The calculation for $\gamma<1/\delta$ proceeds analogously, the only difference being that
								\[
								I_1(s)\lesssim\mathds{1}\{s< r^{-2} \}\,s+\mathds{1}\{r^{-2} \leq s \}\,r^{-2\delta}s^{1-\delta},
								\]
								and that $r^{-2/\gamma}s^{-(1-\gamma)/\gamma}\leq r^{-2\delta}s^{1-\delta}$ if $s\geq r^{-2}$, which yields
								\[
								Q_r(s)\lesssim \mathds{1}\Big\{s<r^{-\frac{2}{\gamma}}\Big\}+\mathds{1}\Big\{r^{-\frac{2}{\gamma}}\leq s < r^{-2}\Big\}\,r^{-\frac{2}{1-\gamma}}s^{-\frac{\gamma}{1-\gamma}}+\mathds1\big\{r^{-2}\leq s\big\} \,r^{-2\delta}s^{1-\delta},
								\] 
								in this case.
						\end{proof}
						%\begin{remark}
						%	We remark that the calculation works mutatis mutandis in any dimension {\color{cyan}(with $r$ replaced by $r^d$)}. In particular, if $d=2$, then the connection probability is $O(r^{-4})$ at distance $r$, but this is not sufficient to apply Proposition ~\ref{thrm:recurrence_main}. In all dimensions, the preferential attachment kernel is special in that it creates marginal connection probabilities that behave as in scale invariant long-range percolation models for a wide range of $\gamma$, namely $P_r\asymp r^{-2d}$. A more detailed discussion of this phenomenon is provided in \cite{GracaLuchtMonch22}. 
						%\end{remark}
						The reason for our study of $Q_r(s)$ is the following result, which provides sufficient conditions to apply Theorem~\ref{thrm:cauchy}. By Lemma~\ref{lemma:projected} (1), we have that almost surely
						\begin{equation}\label{eq:DefKConductance}
						C_e \leq \lim_{N\to\infty} C^{(N)}_e<\infty,
						\end{equation}
						 where
						\[
						C^{(N)}_e = \sum_{\ell =1}^N \ell \sum_{\{x,y\}\in \Pi_\ell(e)}\mathds{1}\{\{x,y\}\in E(\mathcal{H})\}, \quad N\geq 1,
						\]
						as long as the point-to-point connection probabilities $P_{|x-y|}$ are {of order at most $|x-y|^{-\alpha}$ for some $\alpha>3$}.
						\begin{proposition}\label{prop:cauchytailsSI}
							Let $\X_\infty$ be a unit intensity Poisson process on \(\R^2\) with an additional point at the origin and random graph on this point process in which points $x,y\in \X_\infty$ are connected with probability \(P_{|x-y|}\) where $(P_r)_{r>0}$ is such that \[\limsup_{r\to\infty}{r^4 P_r<\infty} \text{ and }\limsup_{N\to\infty} \frac{\E \big(C^{(N)}_e\big)^2}{N}<\infty \text{ for any }e\in E(\Z^d).\] Then the connected component of this graph is almost surely recurrent.
						\end{proposition}
						Proposition~\ref{prop:cauchytailsSI} is the reason why we need the restriction $\gamma<1/3$ for the preferential kernel in Theorem~\ref{thmRecurTrans}. As the calculation at the end of this section shows, the sufficient condition in Proposition~\ref{prop:cauchytailsSI} does not hold for this kernel if $\gamma\geq 1/3$.
						\begin{proof}[Proof of Proposition~\ref{prop:cauchytailsSI}]
							The statement follows from Theorem~\ref{thrm:cauchy} upon verifying the Cauchy-tail property for the projected conductances $C_e$. Fix a large integer $n$ and consider the event
							\[
							E_n=\Big\{\exists \{x,y\}\in \bigcup_{\ell \geq n}\Pi_\ell(e) : \{x,y\}\in E(\mathcal{H}) \Big\}.
							\]
							Arguing precisely as in the proof of Lemma~\ref{lemma:projected}(1) and using that $P_r\lesssim r^{-4}$, we can find $K_1,K_2<\infty,$ such that
							\[
							\P(E_n)\leq K_1 \sum_{\ell \geq n} \ell^2 P_\ell\leq K_2 n^{-1}.
							\] Similarly, we can bound
							\[
							\E C_e^{(n)}\leq K_1 \sum_{l\leq n} l^3 P_l \leq K_2 \log n
							\]
							for suitable $K_1,K_2<\infty.$ It follows that %\green{with $K$ from \eqref{eq:DefKConductance},}
							\[
							\P(C_e > n )\leq \P\Big(C_e^{(n)}>{n}\Big) + \P(E_{n}) \lesssim \P\Big(|C_e^{(n)}-\E C_e^{(n)}|^2 >\Big|{n}-\E C_e^{(n)}\Big|^2\Big)+ n^{-1},
							\]
							and $\E C_e^{(n)}\lesssim \log n$ together with the assumption on the second moment $\E (C^{(N)}_e)^2$ and Markov's inequality imply the existence of some constant $C<\infty$ such that
							\[
							\P\Big(|C_e^{(n)}-\E C_e^{(n)}|^2 >\Big|n-\E C_e^{(n)}\Big|^2\Big)\leq \P\Big(|C_e^{(n)}-\E C_e^{(n)}|^2 >\Big(\frac{n}{2}\Big)^2\Big) \leq \frac {4C} n, 
							\]
							which establishes the Cauchy tail.
						\end{proof}
We can now conclude the proof of the recurrence result for the preferential attachment kernel. 
\begin{proof}[{Proof of Theorem~\ref{thmRecurTrans}, recurrence for preferential attachment kernel if $d=2$}] We need \linebreak to verify the assumption of Proposition~\ref{prop:cauchytailsSI}. 
From Lemma~\ref{lem:connprobcorr}, it follows that for any $r>1$,
						\[
						P_r=\int_0^1 Q_r(s) \,\textup{d}s \lesssim r^{-\frac{2}{\gamma}}+ r^{-\frac2\gamma(1-\frac{\gamma}{1-\gamma})-\frac{2}{1-\gamma}} + r^{-2(1-\frac{1-\gamma}{\gamma})-\frac{2}{\gamma}} \leq 3 r^{-4},
						\]
						where we have used that $\gamma<1/2$ to evaluate the integral. Note that we only need to consider the case that $\gamma>1/\delta$, since $P_r$ is non-decreasing in $\gamma$ for each $r$. For the derivation of an upper bound for the second moment of $\E C_e^{(N)}$, note that by translation invariance we may restrict our attention to the bond $e_0:=\{(0,0),(1,0)\}$. To lighten notation, we write $C^{(N)}:=C_{e_0}^{(N)}$ and $\Pi_\ell:=\Pi_\ell(e_0), \ell\in\N$, for the remainder of the proof and use the variables $e,f$ to denote generic potential edges in $\eta^{[2]}$. We now calculate
						\begin{equation}\label{eq:C^Nbd}
							\begin{aligned}
								\E(C^{(N)})^2 & \lesssim \E\Big( \sum_{\ell=1}^N \sum_{m=1}^N \ell m \sum_{(e,f)\in \Pi_\ell\times \Pi_m}\mathds{1}\{e,f\in E(\mathcal{G}^\beta)\} \Big)\\
								& \leq 2 \E\Big( \sum_{\ell=1}^N \sum_{m=1}^\ell \ell m \sum_{(e,f)\in \Pi_\ell \times \Pi_m}\mathds{1}\{e,f\in E(\mathcal{G}^\beta)\} \Big)\\
								& = 2 \sum_{\ell=1}^N \sum_{m=1}^{\ell-1} \ell m \Big[\E \Sigma^{\ell,m}_{0}+ \E \Sigma^{\ell,m}_{1}\Big] +2 \sum_{\ell=1}^N \ell^2 \E \Sigma^{\ell}.
							\end{aligned}
						\end{equation}
						Note that we have used that the contribution from terms with $\ell>m$ is in distribution the same as from terms with $\ell<m$ in the second step.	The three random variables $\Sigma^{\ell,m}_{0},\Sigma^{\ell,m}_{1}$ and $\Sigma^{\ell}$ are given as 
						\begin{eqnarray*}
							\Sigma^{\ell,m}_{0} & = & \sum_{\substack{(e,f)\in \Pi_\ell\times\Pi_m:\\ e\cap f=\emptyset}}\mathds{1}\{e,f\in E(\mathcal{G}^\beta)\},\\
							\Sigma^{\ell,m}_{1} & = & \sum_{\substack{(e,f)\in \Pi_\ell\times\Pi_m:\\ |e\cap f|=1}}\mathds{1}\{e,f\in E(\mathcal{G}^\beta)\},\\
							\Sigma^\ell 	& = & \sum_{e\in\Pi_\ell}\mathds{1}\{e\in E(\mathcal{G}^\beta)\},
						\end{eqnarray*}
						respectively. We now bound the expectations of these random variables using \eqref{eq:distanceconversion} and \eqref{eq:exppicounts}. We obtain
						\begin{equation}\label{eq:firstsum}
							\E\Sigma^\ell  \lesssim  \ell^2\, P_\ell,
						\end{equation}
						as well as 
						\begin{equation}\label{eq:secondsum}
							\begin{aligned}
								\E \Sigma^{\ell,m}_{0} & = \E \Big(\sum_{e\in\Pi_\ell}\mathds{1}\{e\in E(\mathcal{G}^\beta)\} \sum_{\substack{f\in\Pi_m:\\ f\cap e=\emptyset}}\mathds{1}\{f\in E(\mathcal{G}^\beta) \}\Big) \\
								& \leq\E \Big(\sum_{e\in\Pi_\ell}\mathds{1}\{e\in E(\mathcal{G}^\beta) \}\Big) \E \Big(\sum_{e\in\Pi_m}\mathds{1}\{e\in E(\mathcal{G}^\beta) \}\Big)\\
								& \lesssim \ell^2\, P_\ell \; m^2\, P_m,
							\end{aligned}
						\end{equation}
						where we have used that edges which have no vertex in common are sampled independently of each other. Finally, 
						\begin{equation}\label{eq:thirdsum}
							\begin{aligned}
								\E \Sigma^{\ell,m}_{1} & \lesssim \E \Big(\sum_{\{x,y\}\in\Pi_\ell}\sum_{\substack{z\in\eta:\\ m/d_1<|z-y|\leq m/d_2}}\mathds{1}\big\{\{\{x,y\},\{y,z\}\}\subset E(\mathcal{G}^\beta) \big\}\Big) \\
								& \lesssim  \ell^2\,m\, P_{\ell,m},
							\end{aligned}
						\end{equation}
						where we used the notation 
						\[
						P_{\ell,m}=\int_{0}^{1} Q_\ell(s)Q_m(s)\,\textup{d}s
						\] 
						for the probability that two potential edges of length $\ell$ and $m$, respectively, with one vertex in common are present in $E(\mathcal{G}^\beta)$. Inserting \eqref{eq:firstsum}--\eqref{eq:thirdsum} into \eqref{eq:C^Nbd} yields
						
						\[
						\E(C^{(N)})^2 \lesssim \sum_{\ell=1}^N \ell^3 \sum_{m=1}^{\ell-1} m^2\, P_{\ell,m}   + \sum_{\ell=1}^N \ell^3\, P_\ell \sum_{m=1}^{\ell-1}  m^3\, P_m  +\sum_{\ell=1}^N \ell^4\,P_\ell,
						\] 
						hence to show that $\E(C^{(N)})^2\lesssim N$, it remains to show that the first sum is of order $N$, since it follows from $P_\ell\lesssim \ell^{-4}$ that the second and third sum are at most of order $N$ for any $\gamma<1/2$. Equivalently, we may verify that $\sum_{m=1}^{\ell-1} m^2\, P_{\ell,m}\lesssim \ell^{-3}$, which is the final part of our argument.\\
						
						From now on, we work under the assumption that $\gamma<1/3$. We first consider the case $\delta>3$. We only need to treat $\gamma\in(1/\delta,1/3)$, for if we show that $\sum_{m=1}^{\ell-1} m^2\, P_{\ell,m}\lesssim \ell^{-3}$ for such $\gamma$ it also holds for $\gamma\leq 1/\delta$ by monotonicity. Lemma~\ref{lem:connprobcorr} yields, for $\ell\geq m$ and~$\gamma<1/3$, that
						\begin{align*}
							P_{\ell,m}=\int_0^1 Q_\ell(s)Q_m(s) \lesssim\,  \ell^{-\frac{2}{\gamma}} & + \mathds{1}\{m<\ell^\gamma\} \, \ell^{-\frac{2}{\gamma}} \int_{\ell^{-2}}^{m^{-2/\gamma}}s^{-\frac{1-\gamma}{\gamma}}\,\textup{d}s \\
							& + \ell^{-\frac{2}{\gamma}}m^{-\frac{2}{1-\gamma}} \int_{\ell^{-2}\vee m^{-2/\gamma}}^{m^{-2}}s^{-\frac{1-\gamma}{\gamma}-\frac{\gamma}{1-\gamma}}\,\textup{d}s\\
							& + \ell^{-\frac{2}{\gamma}}m^{-\frac{2}{\gamma}} \int_{m^{-2}}^1 s^{-\frac{2-2\gamma}{\gamma}}\,\textup{d}s\\
							& + \ell^{-\frac{2}{1-\gamma}}\int_{\ell^{-2/\gamma}}^{\ell^{-2}\wedge m^{-2/\gamma}} s^{-\frac{\gamma}{1-\gamma}}\,\textup{d}s\\
							& + \mathds{1}\{m\ge\ell^\gamma\}\ell^{-\frac{2}{1-\gamma}}m^{-\frac{2}{1-\gamma}}\int_{m^{-2/\gamma}}^{ \ell^{-2}} s^{-\frac{2\gamma}{1-\gamma}}\,\textup{d}s.
						\end{align*}
						The sum over the first term $\ell^{-2/\gamma}$ is of the desired order, since $$\sum_{m=1}^{\ell-1} m^2\, \ell^{-2/\gamma}\asymp \ell^{3-2/\gamma}\leq \ell^{-3},$$ for any $\gamma\leq1/3.$ The remaining terms are bounded by a constant times
						\begin{align*}
							\mathds{1}& \{m<\ell^\gamma\} \Big( \ell^{-\frac{2}{\gamma}-2(1-\frac{1-\gamma}{\gamma})} + \ell^{-\frac{2}{\gamma}}m^{-\frac{2}{1-\gamma}-\frac{2}{\gamma}(3-\frac{1}{\gamma(1-\gamma)})}+ \ell^{\frac{2\gamma-2}{1-\gamma}-2} \Big)\\
							& + \mathds{1}\{m\geq \ell^\gamma\}\Big(\ell^{-\frac{2}{\gamma}-6+\frac{2}{\gamma(1-\gamma)}}m^{-\frac{2}{1-\gamma}}+ \ell^{-\frac{2}{1-\gamma}}m^{\frac{1}{1-\gamma}(4-\frac{2}{\gamma})} +  m^{-\frac{2}{1-\gamma}}\ell^{-2+\frac{4\gamma-2}{1-\gamma}} \Big)\\
							& +\ell^{-\frac{2}{\gamma}}m^{-\frac{2}{\gamma}-6+\frac4\gamma} 
							\\
							= & \, \mathds{1}\{m<\ell^\gamma\} \Big(2 \ell^{-4} + \ell^{-\frac{2}{\gamma}}m^{\frac{2}{\gamma^2(1-\gamma)}-\frac{2}{1-\gamma}-\frac{6}{\gamma}} \Big)\\
							& + \mathds{1}\{m\geq \ell^\gamma\}\Big(\ell^{-4}m^{-\frac{2}{1-\gamma}}+ \ell^{-\frac{2}{1-\gamma}}m^{\frac{4\gamma-2}{\gamma(1-\gamma)}} + \ell^{\frac{4\gamma-2}{1-\gamma}-2}m^{-\frac{2}{1-\gamma}}\Big)\\
							& +\ell^{-\frac{2}{\gamma}}m^{\frac{2}{\gamma}-6}.
						\end{align*}
						Multiplying by $m^2$ and summing over $m$ yields
						$$ \sum_{m=1}^{\ell-1} m^2\, P_{\ell,m}\lesssim
						\ell^{3\gamma-4} + \ell^{-1-\frac{2}{1-\gamma}}+ \ell^{-3}
						\le 3\ell^{-3},$$						where we have used that $\gamma<1/3$ for the second term of the sum as well as in the final bound. 
						%This concludes the argument, since the first two terms are of order at most $\ell^{-3}$ whenever $\gamma<1/3$.
						\\
						
						Let us finalise the proof by dealing with the case $\delta\leq 3$. In this case, necessarily, $\gamma<1/\delta$, due to our standing assumption that $\gamma<1/3$, and we need to apply the second case in Lemma~\ref{lem:connprobcorr} to obtain an estimate for $\sum_{m=1}^{\ell-1} m^2\, P_{l,m}$. Repeating the steps above yields
						\begin{align*}
							P_{l,m}\lesssim  \ell^{-\frac{2}{\gamma}} & +\ell^{-2\delta}m^{2\delta-6}\\
							& + \mathds{1}\{m<\ell^\gamma\} \Big(\ell^{-4} + \ell^{-2\delta}m^{\frac{1}{\gamma}(2\delta-4)} \Big)\\
							& +\mathds{1}\{m\geq \ell^\gamma\}\Big(\ell^{-4}m^{-\frac{2}{1-\gamma}}+\ell^{\frac{2\gamma}{1-\gamma}-4}m^{-\frac{2}{1-\gamma}}+\ell^{-\frac{2}{1-\gamma}}m^{\frac{4\gamma-2}{\gamma(1-\gamma)}} \Big).
						\end{align*}
						We only need to consider those summands of $\sum_{m=1}^{\ell-1} m^2\, P_{\ell,m}$ that did not appear already in the previous calculation, for these we get as upper bound a constant times 
%						We only need to bound \red{the sum $\sum_{m=1}^{\ell-1} m^2 (\cdot)$} over terms which did not appear already in the previous calculation (mind that the exponents obtained there were independent of the choice of $\delta>0$), which yields an expression of order at most
						\[
						\ell^{-3}+\ell^{-6+3\frac\gamma{1-\gamma}} + \ell^{3\gamma-4} \le 3\ell^{-3}, 
						\]
						because $\gamma<1/3$. The proof is complete.
					\end{proof}
					
					\appendix
					\section{Elements of electrical network theory}\label{sect:aux}
					%\subsection{Calculations using electrical network theory}\label{app:eclectic}
					In this section, we provide some well-known results for reference. %In particular, we prove Lemmas \ref{thrm:recurrence1} and \ref{thrm:recurrence2} below by adapting the arguments of Berger~\cite{Berge02} from the lattice to the continuum case.
					The basic framework used behind the scenes of both our recurrence and transience proofs is that of electrical {network} theory. We restrict ourselves to recalling some basic results that we need in our proofs, a comprehensive treatment of the theory can be found for instance in \cite{LyonsPeres17}. A connected loop-free multigraph $G=(V(G),E(G))$ together with a \emph{conductance} function $C\colon E(G)\to (0,\infty)$ is called a \emph{network}. Note that we may always view $C$ as a function defined on $V(G)^{[2]}$ setting $C(e)=0$ for potential edges $e\notin E(G)$. The random walk $Y=(Y_i)_{i\geq 0}$ on $(G,C)$ is obtained by reweighing the transition probabilities of simple random walk on $G$ according to $C$, i.e.\ the walker chooses their way with probabilities proportional the sum of the conductances on the edges incident to their current position. In particular, we obtain simple random walk 
					on~$G$ as a special case, {if $C$ is constant}. We consider a locally finite network~$G$,~i.e.\ 
					$$\pi(x):=\sum_{e\in E(G): e \text{ incident to }x}C(e)<\infty\quad \text{for all }x\in V(G),$$ 
					then $\pi$ is an invariant measure for $Y$. Let further $\mathsf{P}^G(v\to Z)$ denote the probability that $Y$ visits $Z\subset V(G)$ before returning to $v\in V(G)$ when started in $Y_0=v\in V(G).$ To characterise recurrence and transience of $Y$ in infinite networks it is convenient to define the \emph{effective conductance} between $v\in V(G)$ and $Z\subset V(G)$ as
					\[
					C_{\textup{eff}}(v,Z)=C_{\textup{eff}}^G(v,Z)=\pi(v)\mathsf{P}^G(v\to Z),
					\]
					for finite $G$ and then extend the notion to infinite graphs via a limiting procedure. In particular, by identifying all vertices at graph distance further than $n$ from $v\in V(G)$ with one vertex $z_n$ (whilst removing any loops and keeping multiple edges with their conductances) we obtain a sequence of finite networks $(G_n,C_n)$. Moreover, the limit
					\[
					C_{\textup{eff}}^G(v,\infty)=\lim_{n\to\infty}C^{G_n}_{\textup{eff}}(v,z_n)\in[0,\infty)\quad\mbox{ for } v\in V(G),
					\]
					is well-defined. The following characterisation of recurrence and transience is classical, see e.g.\ {\cite[Theorem 2.3]{LyonsPeres17}}.  
					\begin{theorem}\label{thrm:effres}
						$Y$ is transient if and only if
						\[
						C_{\textup{eff}}(v,\infty)>0 \quad\text{ for some }v\in V(G).
						\]
						Morevover, if $C_{\textup{eff}}(v,\infty)>0$ for some $v\in V(G)$, then $C_{\textup{eff}}(v,\infty)>0$ for all $v\in V(G)$.
					\end{theorem}
					In particular, a vanishing upper bound on the effective conductance is sufficient for recurrence. Let $v\in G$ be fixed. Recall that a \emph{cutset} $\Pi$ (for $v$) in $G$ is a set of edges that separates $v$ from~$\infty$, i.e. the connected component of $v\in (V(G), E(G)\setminus\Pi)$ is finite.
					%\MH{Perhaps say what we add to Berger's arguments?}
					%As in \cite{Berge02}, the proof of \Cref{thrm:recurrence1} relies on a result of Nash and Williams, which we state here for completeness.
					\begin{theorem}[Nash-Williams, see e.g.\ \cite{Peres99}]\label{thrm:nash1}
						Let $Y_0=v$ and \(\{\Pi_n\}_{n=1}^{\infty}\) be disjoint cutsets for~$v$. Denote by \(C_{\Pi_n}\) the sum of the conductances of edges in \(\Pi_n\). Then 
						\[
						C_{\textup{eff}}(v,\infty)\leq \left(\sum_nC_{\Pi_n}^{-1}\right)^{-1}.
						\]
						In particular, if
						\[
						\sum_nC_{\Pi_n}^{-1}=\infty,
						\]
						then the random walk $Y$ is recurrent.
					\end{theorem}
					We close this section by stating a few rules for calculating effective conductances, which are used in many instances throughout the paper.
					\begin{theorem}[Electrical network calculus, see e.g.\ \cite{LyonsPeres17}]\label{thrm:monotonicity}
						Let $G$ be a multigraph and let $C,C'$ be conductances on $E(G)$.
						\begin{itemize}
							\item (Parallel Law) Replacing two parallel edges (i.e.\ two edges adjacent to the same pair of vertices) $e_1,e_2\in E(G)$ by a single edge with conductance $C(e_1)+C(e_2)$ leaves the effective conductance $C_{\textup{eff}}$ unchanged.
							\item (Series Law) If a vertex $x\in V(G)$ of degree two and its incident edges $(x,y),(x,z)$ (with $y\neq z$) are replaced by a single edge $(y,z)$ with conductance $(C(x,y)^{-1}+C(x,z)^{-1})^{-1}$, then the effective conductance $C_{\textup{eff}}(v,w)$ remains unchanged (for all $v,w\neq {x}$).
							\item (Monotonicity Principle) If $C(e)\leq C'(e)$ for all $e\in E(G)$, then
							\[
							C_{\textup{eff}}(v,Z)\leq C'_{\textup{eff}}(v,Z)\quad \text{for all }v\in V(G), Z\subset V(G).
							\]
							\item (Short-circuiting) {If $x,y\in V(G)$ are collapsed into one vertex and all edges are retained then the effective conductances between all remaining vertices do not decrease.}
							\item (Flow-cut inequality) If $v\in V(G)$ and $Z,Z'\subset V(G)$ are disjoint and such that $v$ and $Z$ are disconnected in $G$ upon removing all edges incident to $Z'$, then
							\[C_{\textup{eff}}(v,Z)\leq C_{\textup{eff}}(v,Z').\]
						\end{itemize}
					\end{theorem}
					%In particular, it follows from Theorem~\ref{thrm:monotonicity}, that
					%\begin{itemize}
					% 	\item adding edges to a graph increases effective conductance (since it corresponds to raising an edge conductance from $0$ to $1$);
					% 	\item collapsing two vertices into one whilst keeping all edges increases effective conductance (since it corresponds to adding an edge with infinite conductance).
					% \end{itemize}

\printbibliography

% add below the content of your .bbl file produced by bibtex.

%%%%%%%%%%%%%%%%%%%%%%%%%%%%%%%%%%%%%%%%%%%%%%%%%%%%%%%%%%%%%%%%%%%
%%                                                               %%
%% You may add acknowledgments (optional).                       %%
%%                                                               %%
%%%%%%%%%%%%%%%%%%%%%%%%%%%%%%%%%%%%%%%%%%%%%%%%%%%%%%%%%%%%%%%%%%%
\begin{acks}
We thank Noam Berger for correspondence concerning the recurrence results.
\end{acks}

%%%%%%%%%%%%%%%%%%%%%%%%%%%%%%%%%%%%%%%%%%%%%%%%%%%%%%%%%%%%%%%%%%%
%%                                                               %%
%% You have reached the end of your document.                    %%
%%                                                               %%
%%%%%%%%%%%%%%%%%%%%%%%%%%%%%%%%%%%%%%%%%%%%%%%%%%%%%%%%%%%%%%%%%%%

\end{document}